\documentclass[10 pt, letterpaper]{article}
\usepackage[T1]{fontenc}
\usepackage{times}
\renewcommand{\baselinestretch}{1}
\usepackage{xspace,eurosym,verbatim}
\usepackage{amsmath,amssymb,bm,upgreek,theorem}
\usepackage{float,graphicx,epsfig,pstricks,pst-node,subfig}
\usepackage{ifpdf,cite}
\usepackage{multirow}

\psset{unit=1mm,framesep=10pt,fillstyle=none}  
\degrees[360]                                  
\psset{linewidth=0.9pt}                        
\providecommand{\thl}{\psset{linewidth=0.3pt}} 
\providecommand{\mel}{\psset{linewidth=0.5pt}} 
\providecommand{\stl}{\psset{linewidth=0.9pt}} 
\providecommand{\bol}{\psset{linewidth=1.4pt}} 

\numberwithin{figure}{section}
\numberwithin{table}{section}
\numberwithin{equation}{section}

\newtheorem{theorem}{Theorem}[section]

\newtheorem{lemma}[theorem]{Lemma}

\newtheorem{proposition}[theorem]{Proposition}

\newtheorem{definition}[theorem]{Definition}
\newtheorem{remark}[theorem]{Remark}
\newtheorem{assumption}[theorem]{Assumption}
\newtheorem{example}[theorem]{Example}
\newenvironment{proof}{\vspace{1.5ex}\textit{Proof.}\rmfamily}{\hfill$\square$\vspace{1.5ex}}

\providecommand{\eg}{e.g.\@\xspace} %
\providecommand{\ie}{i.e.\@\xspace} %
\providecommand{\cf}{cf.\@\xspace} %
\providecommand{\st}{\textnormal{s.t.}}
\providecommand{\pst}{\phantom{\textnormal{s.t.}}}
\providecommand{\iid}{i.i.d.\@\xspace} %
\providecommand{\ef}{\enspace.}
\providecommand{\ec}{\enspace,}
\providecommand{\esp}{\enspace}

\providecommand{\fa}{\forall\,\,} %
\providecommand{\tp}{^{\textrm{T}}} %
\providecommand{\to}{\rightarrow}

\providecommand{\xo}{x^{\star}} %
\providecommand{\bxo}{\bar{x}^{\star}} %
\providecommand{\ep}{\varepsilon} %
\providecommand{\de}{\delta} %
\providecommand{\bde}{\bar{\delta}} %
\providecommand{\den}{\delta} %
\providecommand{\om}{\omega} %
\providecommand{\bom}{\bar{\omega}} %
\providecommand{\tom}{\tilde{\omega}} %
\providecommand{\Om}{\Delta^{K}} %
\providecommand{\Omi}{\Delta^{K_{i}}} %
\providecommand{\tOmi}{\Delta^{K_{i}-R_{i}}} %
\providecommand{\sd}{\zeta} %
\providecommand{\sdi}{\zeta_{i}} %
\providecommand{\sr}{\rho} %
\providecommand{\sri}{\rho_{i}} %
\providecommand{\V}{V}
\providecommand{\Vb}{\bar{V}}

\DeclareMathOperator{\esssup}{ess\,sup} %
\DeclareMathOperator{\argmin}{arg\,min} %
\DeclareMathOperator{\conv}{conv} %
\DeclareMathOperator{\rnk}{rank}
\DeclareMathOperator{\spn}{span}

\DeclareMathOperator{\Sc}{Sc}

\DeclareMathOperator{\bSc}{\overline{\Sc}}

\DeclareMathOperator{\B}{B}
\DeclareMathOperator{\di}{d}

\providecommand{\CA}{\mathcal{A}}
\providecommand{\CL}{\mathcal{L}}

\providecommand{\BN}{\mathbb{N}}
\providecommand{\BR}{\mathbb{R}}
\providecommand{\bBR}{\overline{\mathbb{R}}}
\providecommand{\BRn}{\mathbb{R}^{n}}
\providecommand{\BRd}{\mathbb{R}^{d}}
\providecommand{\BX}{\mathbb{X}}
\providecommand{\BZ}{\mathbb{Z}}

\DeclareMathOperator{\SCP}{SCP}
\DeclareMathOperator{\MCP}{MCP}
\DeclareMathOperator{\MSP}{MSP}
\providecommand{\DSP}{\overline{\MSP}}
\providecommand{\Pb}{\textstyle{\Pr}}

\title{Randomized Solutions to Convex Programs with Multiple Chance Constraints\thanks{This
manuscript is the preprint of a paper submitted to the SIAM Journal on Optimization and it is
subject to SIAM copyright. SIAM maintains the sole rights of distribution or publication of the work 
in all forms and media. 
If accepted, the copy of record will be available at http://www.siam.org.}}

\author{Georg Schildbach\thanks{Automatic Control Laboratory, Swiss Federal Institute
    of Technology, Zurich, Switzerland
    ({\tt\footnotesize schildbach|morari@control.ee.ethz.ch}).}
    \and Lorenzo Fagiano\thanks{Dipartimento di Automatica e Informatica, Polytecnico di Torino,
    Torino, Italy and Department of Mechanical Engineering, University of California at Santa
    Barbara, Santa Barbara (CA), United States
    ({\tt\footnotesize lorenzo.fagiano@polito.it}).}
    \and Manfred Morari$^{\dagger}$
}

\begin{document}

\maketitle

\begin{abstract}
The scenario-based optimization approach (`scenario approach') provides an intuitive way of
approximating the solution to chance-constrained optimization programs, based on finding the
optimal solution under a finite number of sampled outcomes of the uncertainty (`scenarios').
A key merit of this approach is that it neither requires explicit knowledge of the uncertainty set,
as in robust optimization, nor of its probability distribution, as in stochastic optimization.
The scenario approach is also computationally efficient because it only requires the solution to a
convex optimization program, even if the original chance-constrained problem is non-convex.
Recent research has obtained a rigorous foundation for the scenario approach, by establishing
a direct link between the number of scenarios and bounds on the constraint violation probability.
These bounds are tight in the general case of an uncertain optimization problem with a single
chance constraint.

This paper shows that the bounds can be improved in situations where the chance constraints have a
limited `support rank', meaning that they leave a linear subspace unconstrained.
Moreover, it shows that also a combination of multiple chance constraints, each with individual
probability level, is admissible.
As a consequence of these results, the number of scenarios can be reduced from that prescribed by
the existing theory for problems with the indicated structural property.
This leads to an improvement in the objective value and a reduction in the computational complexity
of the scenario approach.
The proposed extensions have many practical applications, in particular high-dimensional problems such as multi-stage uncertain decision problems or design problems of large-scale systems.\end{abstract}


\section{Introduction}\label{Sec:Intro} 

Optimization is ubiquitous in modern problems found in engineering, logistics, and other sciences. A
common pattern is that a decision or design variable $x\in\BRd$ has to be selected from a subset of
$\BRd$, as described by constraints $f_i:\BRd\to\BR$, and its quality is measured against some
objective or cost function $f_0:\BRd\to\BR$:
\begin{subequations}\label{Equ:OptProb}\begin{align}
    \min_{x\in\BRd}\quad &f_0(x)\ec\\
    \st\quad&f_{i}(x)\leq 0\qquad\fa i=1,2,\hdots,N\ef
\end{align}\end{subequations}

\subsection{Chance-Constrained Optimization}

Unfortunately, in many practical applications the underlying problem data is uncertain.
This uncertainty shall be represented with an abstract variable $\den\in\Delta$, where $\Delta$
is an uncertainty set whose nature is not specified.
The uncertainty may affect the objective function $f_0$ and/or the constraints $f_i$. Thus for a
particular decision $x$ it becomes uncertain what objective value is achieved and/or whether
the constraints are indeed satisfied.
The second situation represents a particular challenge, as good solutions are usually located on
the boundary of the feasible set.

This gives rise to a trade-off problem between the (uncertain) objective value and the robustness
of the chosen decision to a constraint violation.
A large variety of approaches addressing this issue have been proposed in the areas of robust and
stochastic optimization \cite{BaiEtAl:1997,BenTalNem:1998,BirgLouv:1997,KallMay:2011,KouvYu:1997,
MulVan:1995,Prekopa:1995,Shapiro:2009}, with the preferred method of choice depending on the
requirements of the application at hand.

In many practical applications, $\den$ can be assumed to be of a stochastic nature. In this case,
the formulation of \emph{chance constraints}, where the decision variable $x$ has to be feasible
with a least probability $(1-\ep)$ for $\ep\in(0,1)$, has proven to be an appropriate concept
for handling the uncertainty in the constraints.
However, chance-constrained optimization problems are usually very difficult to solve.
The \emph{scenario approach}, as explained below, represents an attractive method for finding an
`approximate solution' to stochastic programs, since it is both intuitive and computationally
efficient.

\subsection{The Scenario Approach}\label{Sec:SCP}

Recent contributions \cite{Cala:2009,CalaCamp:2005,Cala:2010,CampGar:2008,CampGar:2011} have
revealed the theoretical links between the scenario approach and the solution to an optimization
problem with a linear objective function and a single chance constraint ($\SCP$):
\begin{subequations}\label{Equ:SCP}\begin{align}
    \min_{x\in\BX}\quad &c\tp x\ec\\
    \st\quad&\Pr\bigl[f(x,\den)\leq 0\bigr]\geq(1-\ep)\ef
\end{align}\end{subequations}
Here $\BX\subset\BRd$ is a compact and convex set, $c\tp$ denotes the transpose of a vector $c\in
\BRd$, $\Pr[\cdot]$ is the probability measure on the uncertainty set $\Delta$, $f:\BRd\times\Delta
\to\BR$ is a convex function in its first argument $x\in\BRd$ for $\Pr$-almost every uncertainty
$\den\in\Delta$, and $\ep$ is some value in the open real interval $(0,1)$.

The chance constraint (\ref{Equ:SCP}b) is interpreted as follows. For any given $x\in\BRd$, the
left-hand side represents the probability of the event that $x$ indeed belongs to the feasible set.
Written more properly,
\begin{equation}\label{Equ:ProbMeas}
	\Pr\bigl[f(x,\den)\leq 0\bigr]:=\Pr\bigl\{\den\in\Delta\:\big|\:f(x,\den)\leq 0\bigr\}\ec
\end{equation}
however the left-hand side notation is kept throughout for brevity.
Note that $x$ is considered to be a \emph{feasible point} of the chance constraint (\ref{Equ:SCP}b)
if this probability is at least $(1-\ep)$.

\vspace*{0.15cm}
\begin{remark}[Problem Formulation]\label{Rem:Generality}
    The formulation of the $\SCP$ encompasses a vast range of problems, namely any uncertain
    optimization problem that becomes convex if the value of $\den$ were fixed.
    (a) Any uncertain convex objective function $f(\cdot,\den)$ can be included by an epigraph
    reformulation, with the new objective being a scalar and hence linear \cite[Sec.\,3.1.7]
    {BoydVan:2004}.
    (b) Joint chance constraints, where $x$ must satisfy multiple convex constraints simultaneously
    with probability $(1-\ep)$, are covered since the intersection of convex sets is convex.
    (c) Additional deterministic, convex constraints can be included by intersection with the
    compact set $\BX$.
\end{remark}
\vspace*{0.15cm}

The characterization of the feasible set of a chance constraint requires exact knowledge of the
probability distribution of $\den$. Moreover, the feasible set is non-convex and difficult to
express explicitly, except for very special cases \cite{BirgLouv:1997,KallMay:2011,Prekopa:1995,Shapiro:2009}.
This makes the $\SCP$, in full generality and especially in higher dimensions $d$, an extremely
difficult problem to solve.

The scenario approach can be used to find an \emph{approximate solution} to the $\SCP$, which is
considered to be any point in $\BX$ that is feasible for the chance constraint with some given
(very high) \emph{confidence} $(1-\theta)\in (0,1)$.
This problem is usually not as hard, if an approximate solution is chosen in a low-violation region
of the decision space (with high confidence). However, then the resulting objective value may
be poor, in which case the approximate solution shall be called `\emph{conservative}'.
Clearly, it is of major interest to find approximate solutions that are the least conservative (\ie
with an objective value as low as possible), and this is the goal of the scenario approach.

The basic idea of the scenario approach is to draw a specific number $K\in\BN$ of samples
(`\emph{scenarios}') from the uncertainty $\den$, and to take the optimal solution that is feasible
under all of these scenarios (`\emph{scenario solution}') as an approximate solution.
Computing the scenario solution involves a deterministic optimization program (`\emph{scenario
program}'), which is obtained by replacing the chance constraint (\ref{Equ:SCP}b) with the $K$
sampled deterministic constraints.

By construction, the scenario program is a deterministic, convex optimization program that can be
solved efficiently by standard algorithms \cite{BoydVan:2004,LuenYe:2008,NocWri:2006}.
Moreover, the scenario approach is distribution-free in the sense that it does not rely on a
particular mathematical model for the distribution of $\den$, or even its support set $\Delta$.
In fact, both may be unknown; the only requirements are stated in the following assumption.

\vspace*{0.15cm}
\begin{assumption}[Uncertainty]\label{Ass:Uncertainty}
    (a) The uncertainty $\den$ is a random variable with (possibly unknown) probability measure
    $\Pr$ and support set $\Delta$.
    (b) A sufficient number of independent random samples from $\den$ can be obtained.
\end{assumption}
\vspace*{0.15cm}

Note that Assumption \ref{Ass:Uncertainty} is fairly general. It could even be argued that the
scenario approach is at the heart of any robust and stochastic optimization method, because either
the uncertainty set $\Delta$ or the probability distribution of $\den$ are usually constructed based
on some (necessarily finite) experience of the uncertainty.

Tight bounds for the proper choice of the sample size $K$ are established by \cite{Cala:2010,
CampGar:2008}, when linking it directly to the probability with which the scenario solution violates
the chance constraint (\ref{Equ:SCP}b).
Moreover, \cite{Cala:2010,CampGar:2011} show that the theory can be extended to the case where
$R\leq K$ sampled constraints are discarded \emph{a posteriori}, that is after observing the
outcomes of the $K$ samples.
While this increases the complexity of the scenario approach (in terms of data requirement and
computation), it can be used to improve the objective value achieved by the scenario solution.
In fact, the scenario solution can be shown to converge to the exact solution of \eqref{Equ:SCP}
when the number of discarded constraints are increased, given that some mild technical assumptions
hold, \cf \cite[Sec.\,4.4]{CampGar:2011}

\subsection{Novel Contributions}

From a practical point of view, the strongest appeal of the scenario approach is the facility of
its application and the low computational complexity.
It becomes particularly attractive for uncertain optimization problems in higher dimensions,
as these occur frequently in fields such as engineering or logistics.
In these cases, an uncertain constraint will often not involve all decision variables
simultaneously, as allowed by the general case of (\ref{Equ:SCP}b).
Instead, multiple uncertain constraints may be present, each of them involving only a subset of
the decision variables.

\vspace*{0.15cm}
\begin{example}[Multi-Stage Decision Problems]\textnormal{
An important example are uncertain \emph{multi-stage decision problems} \cite[Cha.\,7]{BirgLouv:1997}, \cite[Cha.\,8]{KallMay:2011} \cite[Cha.\,13]{Prekopa:1995} \cite[Cha.\,3]{Shapiro:2009},
which occur in many fields such as production planning, portfolio optimization, or control theory.
The basic setting is that some \emph{decision} (\eg on production quantities, buy/sell
orders, or control inputs) has to be taken repeatedly at a finite number of time steps.
Each decision affects the \emph{state} of the system (\eg inventory level, portfolio, or state
variable) at the subsequent time step.
Besides the decision, the state is also subject to uncertain influences (\eg customer
demand, price fluctuations, or dynamic disturbances).
If constraints on the state variables are present (\eg service levels, value at risk, or safety
regions), this adds multiple uncertain constraints (one for the state of each time step) to the
overall decision problem.
Further deterministic constraints may hold for the decision variables, for example.
The special structure of such a problem is that a constraint on the state at some time step
involves only the decisions made prior to this time step, while the decisions afterwards are
not involved.}
\end{example}
\vspace*{0.15cm}

This paper extends the theory of the scenario approach for problems where a single (or multiple)
chance constraint(s) are present that involve only a subset of the decision variables.
More precisely, the chance constraint(s) may affect only a certain subspace of the decision space,
whose dimension will be called its `\emph{support rank}'.
Other constraints, either deterministic or uncertain, cover the directions that are left
unconstrained, so that the solution remains bounded.

The main result of this paper is that an uncertain constraint with a lower support rank can only
supply a lower number of \emph{support constraints} \cite{Cala:2010,CalaCamp:2005,CampGar:2008},
and therefore its associated sample size can be reduced.
This leads to a subtle shift from the idea of a `\emph{problem dimension}' in the existing theory
to that of a `\emph{support dimension}' of a particular chance constraint.
Moreover, it requires an extension of the existing theory to cope with multiple chance constraints
in the uncertain optimization program.
Finally, the approach of constraint removal \emph{a posteriori} is carried over almost analogously
to this extended setting.

From a practical point of view, these extensions improve on the merits of the scenario approach for
problems that have a structure described above.
In particular, the lower sample sizes reduce the computational complexity of the scenario approach
and simultaneously improve the objective value of the scenario solution.
At the same time, the feasibility guarantees for the scenario solution remain as strong as before.
Hence the extensions of this paper, when applicable, offer only advantages over the existing
results on the scenario approach.

\subsection{Organization of the Paper}

Section \ref{Sec:Problem} contains the problem statement.
Section \ref{Sec:Stucture} introduces some background on its properties, and states the rigorous
definitions for the `\emph{support dimension}' and the `\emph{support rank}' of a chance
constraint.
Section \ref{Sec:ScenSol} contains the main results of this paper, which give the improved sample
bounds in the presence of a single (or multiple) chance constraint(s) of limited support rank.
Section \ref{Sec:DiscardConstr} extends this theory to the sampling-and-discarding procedure, which
can be used to improve the objective value of the scenario solution, at the price of larger data
requirements and an increased computational complexity.
Section \ref{Sec:Example} presents a brief numerical example that demonstrates the application of
the presented theory, as well as its potential benefits when compared to existing results.

\section{Problem Formulation}\label{Sec:Problem}

This section introduces the generalized problem formulation with multiple chance constraints, the
corresponding scenario program, and some basic terminology.

\subsection{Stochastic Program with Multiple Chance Constraints}\label{Sec:MCP}

Consider the following extension of the $\SCP$ to an optimization problem with linear objective
function and multiple chance constraints ($\MCP$):
\begin{subequations}\label{Equ:MCP}\begin{align}
    \min_{x\in\BX}\quad &c\tp x\ec\\
    \st\quad&\Pr\bigl[f_{i}(x,\den)\leq 0\bigr]\geq(1-\ep_{i})
    \qquad\fa i\in\BN_{1}^{N}\ec
\end{align}\end{subequations}
where $i$ is the chance constraint index in $\BN_{1}^{N}:=\{1,2,...,N\}$.
The remarks for the $\SCP$ in Section \ref{Sec:SCP} apply analogously; in particular the following
key assumption is made.

\vspace*{0.15cm}
\begin{assumption}[Convexity]\label{Ass:Convexity}
    The constraint functions $f_i:\BRd\times\Delta\to\BR$ of all chance constraints $i\in
    \BN_{1}^{N}:=\{1,...,N\}$ are convex in their first argument $x\in\BRd$ for $\Pr$-almost every
    $\den\in\Delta$.
\end{assumption}
\vspace*{0.15cm}

Other than Assumption \ref{Ass:Convexity}, the dependence of the functions $f_i(x,\den)$ on the
uncertainty $\den$ is completely generic.

The use of `$\min$' instead of `$\inf$' in (\ref{Equ:MCP}a) is justified by the fact that the
feasible set of a single chance constraint is closed under fairly general assumptions
\cite[Thm.\,2.1]{KallMay:2011}.
This implies that the feasible set of the $\MCP$ is compact, due to the presence of $\BX$, and the
infimum is indeed attained.

It remains a standing assumption that the $\sigma$-algebra of $\Pr$-measurable sets in $\Delta$ is
large enough to contain all sets whose probability is measured in this paper, like the ones in
(\ref{Equ:MCP}b), \cf \cite[p.\,4]{CampGar:2008}.

In order to avoid technical issues, which are of little relevance for most practical applications,
the following is assumed, \cf \cite[Ass.\,1]{CampGar:2008}.

\vspace*{0.15cm}
\begin{assumption}[Existence and Uniqueness]\label{Ass:Uniqueness}
	(a) Problem \eqref{Equ:MCP} admits at least one feasible point. By the compactness of
	$\BX$, this implies that there exists at least one optimal point of \eqref{Equ:MCP}.
    (b) If there are multiple optimal points of \eqref{Equ:MCP}, a unique one is selected by
    the help of a \emph{tie-break rule} (\eg the lexicographic order on $\BRd$).
\end{assumption}
\vspace*{0.15cm}

In principle, an approximate solution to the $\MCP$ can be obtained by the classic scenario
approach.
Namely, a $\SCP$ can be setup with the same objective function (\ref{Equ:SCP}a) as the $\MCP$, and
a chance constraint (\ref{Equ:SCP}b) defined by
\begin{equation}\label{Equ:SingleReform}
	f(x,\den):=\max\bigl\{f_1(x,\den),\hdots,f_N(x,\den)\bigr\}\qquad\text{and}\qquad
	\ep:=\min\bigl\{\ep_{1},\ep_{2},...,\ep_{N}\bigr\}\ef
\end{equation}
Note that $f(x,\den)$ is convex in $x$ for almost every $\den$, since the pointwise maximum of
convex functions is convex.
Any feasible point of this $\SCP$ is also a feasible point of the $\MCP$, and hence an
approximate solution to the $\SCP$ with confidence $(1-\theta)$ is also an approximate solution to
the $\MCP$ with confidence $(1-\theta)$.

However, this procedure introduces a considerable amount of conservatism, because it requires the
scenario solution to simultaneously satisfy \emph{all} constraints $i=1,...,N$ with the
\emph{highest} of all probabilities $(1-\ep_{i})$.
Clearly, this conservatism becomes more severe if the number of chance constraints $N$ is large and
there is a great variation in the values of $\ep_{i}$.

\subsection{The Extended Scenario Approach}\label{Sec:MSP}

The extended scenario approach of this paper can be used to compute an approximate solution of the
$\MCP$, which is a feasible point of every chance constraint $i=1,...,N$ with a given confidence
probability of $(1-\theta_{i})$.
The key difference from the classic scenario approach is that each chance constraint $i\in
\BN_{1}^{N}$ is sampled separately, and with an individual sample size $K_{i}\in\BN$.

Let the \emph{random samples} pertaining to constraint $i$ be denoted $\de^{(i,\kappa_{i})}$, where
$\kappa_{i}\in\{1,...,K_{i}\}$, and for brevity also as the collective \emph{multi-sample} $\om^{(i)}:=\{\de^{(i,1)},...,\de^{(i,K_{i})}\}$.
The collection of all samples is combined in an overall multi-sample $\om:=\{\om^{(1)},...,\om^{(N)}
\}$, with the total number of samples given by $K:=\sum_{i=1}^{N}K_{i}$.
All of these samples can be considered `identical copies' of the random uncertainty $\den$, in the
sense that they are themselves random variables and satisfy the following key assumption.

\newpage 
\begin{assumption}[Independence and Identical Distribution]\label{Ass:Independence}
    The sampling procedure is designed such that the set of all random samples, together with the
    actual random uncertainty,
    \begin{equation*}
        \bigcup_{i\in\BN_{1}^{N}}\bigl\{\de^{(i,1)},...,\de^{(i,K_{i})}\bigr\}
        \cup\bigl\{\den\bigr\}
    \end{equation*}
    form a set of \emph{independent and identically distributed (\iid)} random variables.
\end{assumption}
\vspace*{0.15cm}

The multi-sample $\om$ is an element of $\Delta^{K}$, the $K$-th product of the uncertainty set
$\Delta$, and it is distributed according to $\Pb^{K}$, the $K$-th product of the measure $\Pb$.
The scenario program for multiple chance constraints ($\MSP[\om^{(1)},...,\om^{(N)}]$) is
constructed as follows:
\begin{subequations}\label{Equ:MSP}\begin{align}
    \min_{x\in\BX}\quad & c\tp x\ec\\
    \st\quad  & f_i\bigl(x,\de^{(i,\kappa_{i})}\bigr)\leq 0
    \qquad\fa\kappa_{i}\in\BN_{1}^{K_{i}},\:\:\fa i\in\BN_{1}^{N}\ef
\end{align}\end{subequations}
In problem \eqref{Equ:MSP}, the objective function of the $\MCP$ is minimized, while forcing $x$ to
lie inside the constrained sets for all samples $\de^{(i,\kappa_{i})}$ substituted into the
corresponding constraint $i\in\BN_{1}^{N}$.
Clearly, the solution to problem \eqref{Equ:MSP} is itself a random variable, as it depends on the
random multi-sample $\om$.
For this reason, the scenario approach is a \emph{randomized method} for finding an approximate
solution to the $\MCP$.

Of course, the $\MSP$ is actually solved for the observations of the random samples, leading to its
deterministic instance ($\DSP[\bom^{(1)},...,\bom^{(N)}]$):
\begin{subequations}\label{Equ:DSP}\begin{align}
    \min_{x\in\BX}\quad & c\tp x\ec\\
    \st\quad  & f_i\bigl(x,\bde^{(i,\kappa_{i})}\bigr)\leq 0
    \qquad\fa\kappa_{i}\in\BN_{1}^{K_{i}},\:\:\fa i\in\BN_{1}^{N}\ef
\end{align}\end{subequations}
Note that \eqref{Equ:DSP} arises from \eqref{Equ:MSP} by replacing the \emph{(random) samples}
$\de^{(i,\kappa_{i})}$, $\om^{(i)}$, $\om$ with their \emph{(deterministic) outcomes}
$\bde^{(i,\kappa_{i})}$, $\bom^{(i)}$, $\bom$.
Throughout the paper, these outcomes are indicated by a bar, to distinguish them from the
corresponding random variables.
By Assumption \eqref{Ass:Convexity}, $\DSP$ constitutes a convex program that can be solved
efficiently by a suitable algorithm for convex optimization, \cf \cite{BoydVan:2004,LuenYe:2008,NocWri:2006}.


Note that \eqref{Equ:MSP} remains important for analyzing the (probabilistic) properties of the
(random) scenario solution.
In fact, the subsequent theory is mainly concerned with showing that, with a very high confidence,
the scenario solution is a feasible point of the chance constraints (\ref{Equ:MCP}b), provided that
the sample sizes $K_{1},...,K_{N}$ are appropriately selected.

\subsection{Randomized Solution and Violation Probability}\label{Sec:RandSol}

In order to avoid unnecessary complications, the following technical assumption ensures that there
always exists a feasible solution to the $\MSP$, \cf \cite[p.\,3]{CampGar:2008}.

\vspace*{0.15cm}
\begin{assumption}[Feasibility]\label{Ass:Feasibility}
    (a) For any number of samples $K_{1},...,K_{N}$, the $\MSP$ admits a feasible solution almost
    surely.
    (b) For the sake of notational simplicity, any $\Pr$-null set for which (a) may not hold is
    assumed to be removed from $\Delta$.
\end{assumption}
\vspace*{0.15cm}

Assumption \ref{Ass:Feasibility} can be taken for granted in the majority of practical problems.
When it does not hold in a particular case, a generalization of the presented theory accounting for
the infeasible case can be developed along the lines of \cite{Cala:2010}.

Hence the existence of a solution to $\DSP$ is ensured, and uniqueness holds by Assumption
\ref{Ass:Convexity} and by carry-over of the tie-break rule of Assumption \ref{Ass:Uniqueness}(b),
see \cite[Thm.\,10.1,\,7.1]{Rocka:1970}. Therefore the \emph{solution map}
\begin{equation}\label{Equ:SolMap}
    \bxo:\Delta^{K}\to\BX
\end{equation}
is well-defined, returning the unique optimal point $\bxo(\bom^{(1)},...,\bom^{(N)})$ of the
$\DSP$ for a given outcome of the multi-samples $\{\bom^{(1)},...,\bom^{(N)}\}\in \Delta^{K}$.
The solution map can also be applied to the $\MSP$, for which it is denoted by
$\xo:\Delta^{K}\to\BX$.
Now $\xo(\om^{(1)},...,\om^{(N)})$ represents a random vector of unknown probability
distribution, which is also referred to as the \emph{scenario solution}.
In fact, its distribution is a complicated function of the geometry and the parameters of the
problem.

Note that there are two levels of randomness present in the analysis. The first is introduced by the
random samples in $\om$, which affect the choice of the scenario solution. The second is the actual
random uncertainty $\den$, which determines whether or not the scenario solution is feasible with
respect to the chance constraints (\ref{Equ:MSP}b).
For this reason, the scenario approach presented here is also called a
\emph{double-level-of-probability approach} \cite[Rem.\,2.3]{Cala:2009}.

To highlight the two probability levels more clearly, suppose first that the multi-sample $\bom$ has
already been observed, so that the scenario solution $\bxo(\bom^{(1)},...,\bom^{(N)})$ is fixed.
Then for each chance constraint $i=1,...,N$ in (\ref{Equ:MCP}b), the \emph{a posteriori violation
probability} $\Vb_{i}(\bom^{(1)},...,\bom^{(N)})$ is given by
\begin{equation}\label{Equ:ViolPost}
	\Vb_{i}\bigl(\bom^{(1)},...,\bom^{(N)}\bigr):=
	\Pb\bigl[f_{i}\bigl(\bxo(\bom^{(1)},...,\bom^{(N)}),\den\bigr)>0\bigr]\ef
\end{equation}
In particular, each $\Vb_{i}$ has a deterministic, yet generally unknown, value in $[0,1]$.
If the multi-sample $\om$ has not yet been observed, the scenario solution
$\xo(\om^{(1)},...,\om^{(N)})$ is a random vector and so the \emph{a priori violation probability}
\begin{equation}\label{Equ:ViolPrior}
	\V_{i}\bigl(\om^{(1)},...,\om^{(N)}\bigr):=
	\Pb\bigl[f_{i}\bigl(\xo(\om^{(1)},...,\om^{(N)}),\den\bigr)>0\bigr]
\end{equation}
becomes itself a random variable on $(\Delta^{K},\Pb^{K})$, with support $[0,1]$.
Hence the goal is to choose appropriate sample sizes $K_{1},...,K_{N}$ which ensure that
$\V_{i}(\om^{(1)},...,\om^{(N)})\leq\ep_{i}$ for all $i=1,...,N$, with a sufficiently high
confidence $(1-\theta_{i})$.
Before these results are derived however, some structural properties of scenario programs and
technical lemmas ought to be discussed.

\section{Structural Properties of the Constraints}\label{Sec:Stucture}

In this section, a structural property of a chance constraint is introduced which yields a
reduction in the number of samples below the levels given by the existing theory
\cite{CalaCamp:2005,Cala:2010,CampGar:2008}.
This property relates to the new concept of the \emph{support dimension} or, in a form that is more
easily checked for many practical instances, the \emph{support rank}.

\subsection{Support Constraints}\label{Sec:SupConstr}

The concept of a \emph{support constraint} carries over from the $\SCP$ case, \cf \cite[Def.\,4]{CalaCamp:2005}. An illustration is given in Figure \ref{Fig:SupConstr}.

\vspace*{0.15cm}
\begin{definition}[Support Constraint]\label{Def:SupConstr}
    Consider the $\DSP$ for some outcome of the multi-sample $\bom$.
    (a) For some $i\in \BN_{1}^{N}$ and $\kappa_{i}\in\BN_{1}^{K_{i}}$, constraint $f_{i}(x,\bde^
    {(i,\kappa_{i})})\leq 0$ is a \emph{support constraint} of \eqref{Equ:DSP} if its removal from
    the problem entails a change in the optimal solution:
	\begin{equation*}
		\bxo\bigl(\bom^{(1)},...,\bom^{(N)}\bigr)\neq
		\bxo\bigl(\bom^{(1)},...,\bom^{(i-1)},\bom^{(i)}\setminus\{\bde^{(i,\kappa_{i})}\},
		\bom^{(i+1)},...,\bom^{(N)}\bigr)\ef
	\end{equation*}
    In this case the sample $\bde^{(i,\kappa_{i})}$ is also said `to \emph{generate} this support
    constraint.'
    (b) For each $i\in \BN_{1}^{N}$, the indices $\kappa_{i}$ of all samples that generate a support
    constraint of the $\DSP$ are included in the set $\bSc_{i}$.
    Moreover, the tuples $(i,\kappa_{i})$ of all support constraints of the $\DSP$ are collected in
    the \emph{support (constraint) set} $\bSc$.
    With some abuse of this notation, $\bSc=\bigcup_{i=1}^{N}\bSc_{i}$.
\end{definition}
\vspace*{0.15cm}

Definition \ref{Def:SupConstr}(a) can be stated equivalently in terms of the objective function: a
sampled constraint is a support constraint if and only if the optimal objective function value (or
its preference by the tie-break rule) is strictly larger than when the constraint were removed.
To be more precise, Definition \ref{Def:SupConstr}(b), $\bSc$ may also account for the set $\BX$ as
an additional support constraint. This minor subtlety is tacitly understood in the sequel.

\begin{figure}[t]
    \centering
    \subfloat[b][Two Support Constraints.]{
    \parbox[c][45mm]{42mm}{\centering
    \begin{pspicture}(0,0)(42,45)
        \psline[linewidth=2pt]{->}(15,45)(15,37)
        \rput[bl](16,40.5){$-c$}
        \bol
        \psbezier(3,40)(15,24)(24,18)(40,10)
        \pscustom[linestyle=none,fillstyle=hlines,hatchwidth=0.2pt,hatchsep=3pt,hatchangle=40]{
        	\psbezier(3,40)(15,24)(24,18)(40,10)
			\psline(40,10)(38.2,7.6)
			\psbezier(38.2,7.6)(22.2,15.6)(12.6,22.2)(0.6,38.2)
			}
        \psbezier(39,40)(27,24)(18,18)(2,10)
        \pscustom[linestyle=none,fillstyle=hlines,hatchwidth=0.2pt,hatchsep=3pt,hatchangle=-40]{
        	\psbezier(39,40)(27,24)(18,18)(2,10)
			\psline(2,10)(3.8,7.6)
			\psbezier(3.8,7.6)(19.8,15.6)(29.4,22.2)(41.4,38.2)
			}
		\mel
		\psbezier(2.4,31)(14.4,15)(18,10.4)(30,2.4)
        \pscustom[linestyle=none,fillstyle=hlines,hatchwidth=0.2pt,hatchsep=3pt,hatchangle=40]{
        	\psbezier(2.4,31)(14.4,15)(18,10.4)(30,2.4)
			\psline(30,2.4)(28.4,0)
			\psbezier(28.4,0)(16.4,8)(12,13.4)(0,29.4)
			}	
		\psbezier(39,25)(27,9)(17,6.85)(5,2.85)
        \pscustom[linestyle=none,fillstyle=hlines,hatchwidth=0.2pt,hatchsep=3pt,hatchangle=-40]{
        	\psbezier(39,25)(27,9)(17,6.85)(5,2.85)
			\psline(5,2.85)(5.95,0)
			\psbezier(5.95,0)(17.85,4)(29.4,7.4)(41.4,23.4)
			}
    \end{pspicture}}}
    \subfloat[b][One Support Constraint.]{
    \parbox[c][45mm]{42mm}{\centering
    \begin{pspicture}(0,0)(42,45)
        \pscustom[linestyle=none,fillstyle=hlines,hatchwidth=0.2pt,hatchsep=3pt,hatchangle=-40]{
			\pscurve(0,14.88)(18.12,0)(36.24,14.88)
			}
		\pscustom[linestyle=none,fillstyle=solid,fillcolor=white]{
        	\pscurve(2.12,16)(18.12,3)(34.12,16)
			}
		\pscustom[linestyle=none,fillstyle=hlines,hatchwidth=0.2pt,hatchsep=3pt,hatchangle=-40]{
			\pscurve(4.3,30.05)(23.15,10)(42,30.05)
			}
		\pscustom[linestyle=none,fillstyle=solid,fillcolor=white]{
        	\pscurve(7.15,31)(23.15,13)(39.15,31)
			}	
        \pscustom[linestyle=none,fillstyle=hlines,hatchwidth=0.2pt,hatchsep=3pt,hatchangle=-40]{
			\pscurve(6.15,40.65)(22,16)(37.85,40.65)
			}
		\pscustom[linestyle=none,fillstyle=solid,fillcolor=white]{
        	\pscurve(9,41)(22,19)(35,41)
			}
		\bol
	    \pscurve(9,41)(22,19)(35,41)
	    \mel
	    \pscurve(7.15,31)(23.15,13)(39.15,31)
	    \pscurve(2.12,16)(18.12,3)(34.12,16)
		\psline[linewidth=2pt]{->}(15,45)(15,37)
        \rput[bl](16,40.5){$-c$}
    \end{pspicture}}}
    \subfloat[b][One Support Constraint.]{
    \parbox[c][45mm]{42mm}{\centering
    \begin{pspicture}(0,0)(42,45)
        \psline[linewidth=2pt]{->}(15,45)(15,37)
        \rput[bl](16,40.5){$-c$}
        \bol
        \psbezier(3,40)(15,24)(24,18)(40,10)
        \pscustom[linestyle=none,fillstyle=hlines,hatchwidth=0.2pt,hatchsep=3pt,hatchangle=40]{
        	\psbezier(3,40)(15,24)(24,18)(40,10)
			\psline(40,10)(38.2,7.6)
			\psbezier(38.2,7.6)(22.2,15.6)(12.6,22.2)(0.6,38.2)
			}
		\mel
        \psbezier(39,35)(27,23)(14,14)(2,10)
        \pscustom[linestyle=none,fillstyle=hlines,hatchwidth=0.2pt,hatchsep=3pt,hatchangle=-40]{
        	\psbezier(39,35)(27,23)(14,14)(2,10)
			\psline(2,10)(2.95,7.15)
			\psbezier(2.95,7.15)(14.95,11.15)(29.1,20.9)(41.1,32.9)
			}
		\psbezier(32,42)(27,27)(18,12)(8,4)
        \pscustom[linestyle=none,fillstyle=hlines,hatchwidth=0.2pt,hatchsep=3pt,hatchangle=-40]{
       		\psbezier(32,42)(27,27)(18,12)(8,4)
			\psline(8,4)(9.87,1.66)
			\psbezier(9.87,1.66)(19.87,9.66)(29.85,26.05)(34.85,41.05)
			}
		\psbezier(1.8,25)(13.8,16)(28,6.85)(40,2.85)
        \pscustom[linestyle=none,fillstyle=hlines,hatchwidth=0.2pt,hatchsep=3pt,hatchangle=40]{
        	\psbezier(1.8,25)(13.8,16)(28,6.85)(40,2.85)
			\psline(40,2.85)(39.15,0)
			\psbezier(39.15,0)(27.15,4)(12,13.6)(0,22.6)
			}
    \end{pspicture}}}
    \caption{Illustration of Definition \ref{Def:SupConstr} in $\BR^{2}$. The arrow indicates the
    optimization direction, the bold lines are the \emph{support constraints} of the respective
    configuration.\label{Fig:SupConstr}}
\end{figure}
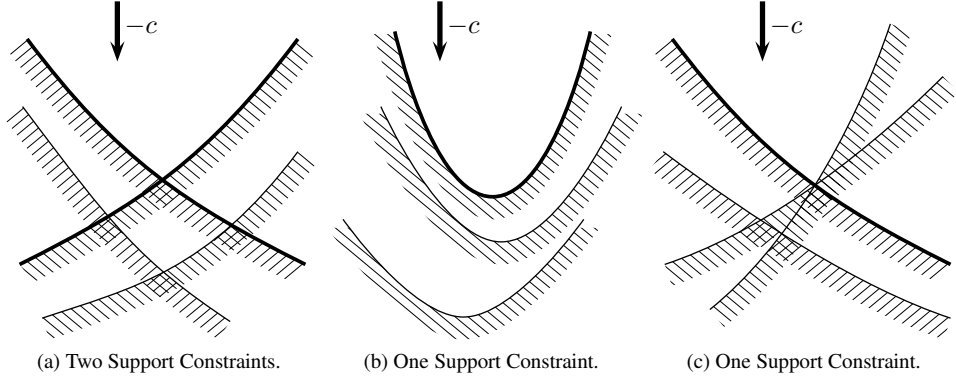

In the stochastic setting of the $\MSP[\om^{(1)},...,\om^{(N)}]$, whether or not a particular random
sample  $\de^{(i,\kappa_{i})}$ generates a support constraint becomes a random event, which can be
associated with a certain probability. Similarly, the support constraint set $\Sc$, and its subsets
$\Sc_{1},...,\Sc_{N}$ contributed by the various chance constraints, are naturally random sets.

\subsection{Support Dimension}\label{Sec:SupDim}

The link between the sample sizes $K_{1},...,K_{N}$ and the corresponding violation probability of
the scenario solution depends decisively on the `dimensions' of the problem.
The following lower bounds represent a mild technical condition, \cf \cite[Thm.\,3.3]{Cala:2010} and
\cite[Def.\,2.3]{CampGar:2008}.

\vspace*{0.15cm}
\begin{assumption}\label{Ass:SampleSize}
	The sample sizes satisfy $K_{1},...,K_{N}\geq d$.
\end{assumption}
\vspace*{0.15cm}

In the existing literature, the dimension of the $\SCP$ has been characterized by \emph{Helly's
dimension}, \cf \cite[Def.\,3.1]{Cala:2010}.
In this paper, there is a subtle shift from the problem dimension to the dimension of chance
constraint $i$ in the $\MCP$, embodied by its \emph{support dimension}.

\vspace*{0.15cm}
\begin{definition}[Support Dimension]\label{Def:SupDim}
    (a) Denote by $|\Sc|$ the (random) cardinality of the set $\Sc$. \emph{Helly's dimension} is the
    smallest integer $\sd$ that satisfies
	\begin{equation*}
        \underset{\omega\in\Delta^{K}}{\esssup}\:|\Sc|\leq\sd\ef
    \end{equation*}
	(b) The \emph{support dimension} of a chance constraint $i\in\BN_{1}^{N}$ in the $\MSP$ is the
	smallest integer $\sdi$ that satisfies
    \begin{equation*}
        \underset{\omega\in\Delta^{K}}{\esssup}\:|\Sc_{i}|\leq\sdi\ef
    \end{equation*}
\end{definition}
\vspace*{0.15cm}

From a basic argument using Helly's Theorem, the number of support constraints $|\Sc|$ of any
(feasible) convex optimization problem in $\BRd$ is upper bounded by the dimension of the decision
space $d$, \cf \cite[Thm.\,2]{CalaCamp:2005}.
This result implies that finite integers $\sd$ and $\sd_{1},...,\sd_{N}$ matching Definition
\ref{Def:SupDim} always exist, so that the concepts of `Helly's dimension' and `support dimension'
are indeed well-defined.
Moreover, the result provides immediate upper bounds on the support dimension of each chance
constraint $i\in\BN_{1}^{N}$ in \eqref{Equ:MSP}, namely $\sdi\leq\sd\leq d$.

It turns out that the support dimension $\sdi$ directly relates to the minimum sample size $K_{i}$ that is required for a given violation level $\ep_{i}$ and residual probability $\theta_{i}$.
The basic mechanism shall be illustrated by the proposition below, for the simpler case of a
\emph{single-level of probability} problem, \cf \cite[Thm.\,1]{CalaCamp:2005}.

\vspace*{0.15cm}
\begin{proposition}[Probability Bound]\label{The:ProbBound}
    Consider a particular constraint $i\in\BN_{1}^{N}$ in the $\MSP[\om^{(1)},...,\om^{(N)}]$ with
    some fixed sample size $K_{i}$, and let $\hat{\sd}_{i}$ be an upper bound for its support
    dimension $\sdi$. Then the following holds:
    \begin{equation}\label{Equ:ProbBound1}
        \Pb^{K+1}\bigl[f_{i}\bigl(\xo(\om^{(1)},...,\om^{(N)}),\den\bigr)>0\bigr]\leq
        \frac{\hat{\sd}_{i}}{K_{i}+1}\ef
    \end{equation}
\end{proposition}

\vspace*{-0.3cm}
\begin{proof}
    Consider $\MSP':=\MSP[\om^{(1)},...,\om^{(i-1)},\om^{(i)}\cup\{\den\},\om^{(i+1)},
    ...,\om^{(N)}]$ and let $\Sc_{i}'\subset\{1,...,K_{i},K_{i}+1\}$ denote the set of support
    constraints generated by samples from $\om^{(i)}\cup\{\den\}$, where $(K_{i}+1)\in\Sc_{i}'$
    stands for $\den$ generating a support constraint.
    Note that the event where $f_{i}\bigl(\xo(\om^{(1)},...,\om^{(N)}),\den\bigr)>0$ can be
    equivalently expressed as $\den$ generating a support constraint of $\MSP'$.
    Hence condition \eqref{Equ:ProbBound1} can be reformulated as
    \begin{equation}\label{Equ:ProbBound2}
        \Pb^{K+1}\bigl[(K_{i}+1)\in\Sc_{i}'\bigr]\leq\frac{\hat{\sd}_{i}}{K_{i}+1}\ef
    \end{equation}

    To analyze the event $(K_{i}+1)\in\Sc_{i}'$, observe that by Assumption \ref{Ass:Independence}
    all samples in $\om^{(i)}\cup\{\den\}$ are \iid, whence all sampled instances of constraint $i$
    in (\ref{Equ:MSP}b) along with `$f_{i}(\,\cdot\,,\den)\leq 0$' are probabilistically identical.
    In particular, they are all equally likely to become a support constraint of $\MSP'$.
    Hence if the number of support constraints $|\Sc_{i}'|$ were known, then
    \begin{equation*}
        \Pb^{K+1}\bigl[(K_{i}+1)\in\Sc_{i}'\bigr]=\frac{|\Sc_{i}'|}{K_{i}+1}\ef
    \end{equation*}
    Even though $|\Sc_{i}'|$ is a random variable, by Definition \ref{Def:SupDim}(b) $|\Sc_{i}'|\leq
    \sdi$ almost surely, and by assumption $\sdi\leq\hat{\sd}_{i}$.
    This immediately yields \eqref{Equ:ProbBound1}.
\end{proof}

\subsection{The Support Rank}\label{Sec:SupRank}

In many practical cases, the support dimension $\sdi$ of a chance constraint $i\in\BN_{1}^{N}$ in
the $\MSP$ is not known exactly. Then it has to be replaced by some upper bound.
As argued above, the existing upper bound is given by the dimension $d$ of the decision space.
However, this bound may not be tight in the case where the constraints satisfy a certain structural
property, namely when they have a limited \emph{support rank}.

Intuitively speaking, the support rank is the dimension $d$ of the decision space less the maximal
dimension of an (almost surely) \emph{unconstrained subspace}.
The latter is understood as a linear subspace of $\BRd$ that cannot be constrained by the sampled
instances of constraint $i$, for almost every value of the multi-sample $\om^{(i)}$.



Before the support rank is introduced in a rigorous manner, three examples of constraint classes
with bounded support rank are described, in order to equip the reader with the necessary intuition
behind this concept. They also show that very common constraint classes possess this property, and
that in practical problems it can often be spotted easily.

\vspace*{0.15cm}
\begin{example}\label{Exa:SupportRank}
    \textnormal{For each of the following cases, a visual illustration can be found in Figure
    \ref{Fig:SupportRank}.}

    (a) Single Linear Constraint. \textnormal{Suppose some chance constraint $i\in\BN_{1}^{N}$ of
    (\ref{Equ:MCP}b) takes the linear form
    \begin{equation}\label{Equ:ExaLinConstr1}
        f_{i}(x,\den)\equiv a\tp x-b(\den)\ec
    \end{equation}
    where $a\in\BR^{d}$, and $b:\Delta\to\BR$ is a scalar depending on the uncertainty in a generic
    way.
    Note that these constraints in the $\MSP$ are unable to constrain any direction in the subspace
    orthogonal to the span of $a$, $\spn\{a\}^{\perp}$, regardless of the outcome of the
    multi-sample $\om^{(i)}$.
    Hence the support rank $\alpha$ of the chance constraint \eqref{Equ:ExaLinConstr1} is equal to
    $1$.}

    (b) Multiple Linear Constraints. \textnormal{As a generalization of case (a), suppose that some
    chance constraint $i\in\BN_{1}^{N}$ of (\ref{Equ:MCP}b) is given by
    \begin{equation}\label{Equ:ExaLinConstr2}
        f_{i}(x,\den)\equiv A(\den) x-b(\den)\ec
    \end{equation}
    where $A:\Delta\to\BR^{r\times d}$ and $b:\Delta\to\BR^{r}$ represent a matrix and a vector that
    depend on the uncertainty $\den$. Moreover, suppose that the uncertainty enters the matrix
    $A(\den)$ in such a way that the dimension of the linear span of its rows $A_{j,\cdot}(\den)$,
    for $j=1,...,r$, satisfies
    \begin{equation*}
        \dim\spn\bigl\{A_{j,\cdot}(\den)\:\big|\:j\in\BN_{1}^{r},\:\den\in\Delta\}\leq\beta<d\ef
    \end{equation*}
    Note that these constraints in the $\MSP$ are unable to constrain any direction in
    $\spn\bigl\{A_{j,\cdot}(\den)\:\big|\:j\in\BN_{1}^{r},\:\den\in\Delta\}^{\perp}$, regardless of
    the outcome of the multi-sample $\om^{(i)}$.
    Hence the support rank of the chance constraint \eqref{Equ:ExaLinConstr2} is equal to $\beta$.}

	(c) Quadratic Constraint. \textnormal{For a nonlinear example, consider the case where some
	chance constraint $i\in\BN_{1}^{N}$ of (\ref{Equ:MCP}b) is given by
    \begin{equation}\label{Equ:ExaLinConstr3}
        f_{i}(x,\den)\equiv\bigl(x-x_{c}(\den)\bigr)\tp Q\bigl(x-x_{c}(\den)\bigr) - r(\den)\ec
    \end{equation}
    where $Q\in\BR^{d\times d}$ is positive semi-definite with $\rnk Q=\gamma<d$, and $x_{c}:\Delta
    \to\BRd$, $r:\Delta\to\BR_{+}$ represent a vector and scalar that depend on the uncertainty.
    Note that these constraints in the $\MSP$ are unable to constrain any direction in the null
    space of the matrix $Q$, regardless of the outcome of the multi-sample $\om^{(i)}$.
    Since this null space has dimension $d-\gamma$, the support rank of the chance constraint
    \eqref{Equ:ExaLinConstr3} is equal to $\gamma$.}
\end{example}
\vspace*{0.15cm}

To introduce the support rank in a rigorous manner, pick a chance constraint $i\in\BN_{1}^{N}$ of
the $\MCP$.
For each point $x\in\BX$ and each uncertainty $\de\in\Delta$, denote the corresponding level set of
$f_{i}:\BRd\times\Delta\to\BR$ by
\begin{equation}\label{Equ:LevelSets}
    F_{i}(x,\de):=\bigl\{\xi\in\BRd\:\big|\:f_{i}(x+\xi,\de)=f_{i}(x,\de)\bigr\}\ef
\end{equation}

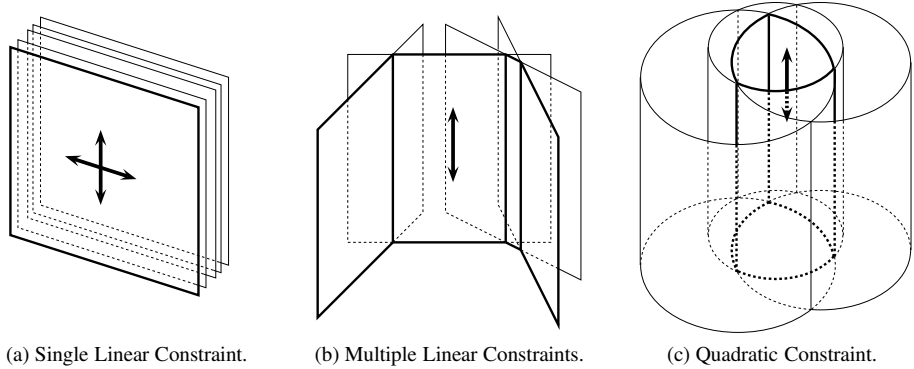
\begin{figure}[H]
    \centering
    \subfloat[b][Single Linear Constraint.]{
    \parbox[c][45mm]{42mm}{\centering
    \begin{pspicture}(3,0)(45,45)
        \stl
        \pspolygon(8,13)(8,38)(33,30)(33,5) 
        \thl
        \psline(9,37.6)(9,39)(34,31)(34,6)(32.6,6.45) 
        \psline(10.3,38.7)(10.3,40.3)(35.3,32.3)(35.3,7.3)(33.7,7.81) 
        \psline(11,40)(11,41)(36,33)(36,8)(35,8.32) 
        \psline(12,40.8)(12,42)(37,34)(37,9)(35.8,9.38) 
        \psline[linestyle=dashed,dash=1pt 1pt](32.6,6.45)(9,14)(9,37.6)
        \psline[linestyle=dashed,dash=1pt 1pt](33.7,7.81)(10.3,15.3)(10.3,38.7) 
        \psline[linestyle=dashed,dash=1pt 1pt](35,8.32)(11,16)(11,40) 
        \psline[linestyle=dashed,dash=1pt 1pt](35.8,9.38)(12,17)(12,40.8) 
        \bol
        \psline{<->}(20,17)(20,27)
        \psline{<->}(15.25,23.5)(24.77,20.45)
    \end{pspicture}}}
    \subfloat[b][Multiple Linear Constraints.]{
    \parbox[c][45mm]{42mm}{\centering
    \begin{pspicture}(0,3)(42,48)
        \stl
        \pspolygon(35,4)(30,14)(28,15)(13,15)(3,5)(3,30)(13,40)(28,40)(30,39)(35,29)
        \psline(30,14)(30,39) 
        \psline(28,15)(28,40) 
        \psline(13,15)(13,40) 
        \thl
        \psline(27,41)(27,45)(30,39) 
        \psline(20,40)(20,44)(28,40) 
        \psline(35,11.5)(38,10)(38,35)(30,39) 
        \psline(34,37)(34,40)(28,40) 
        \psline(7,34)(7,40)(13,40)   
        \psline(17,40)(17,44)(13,40) 
        \psline[linestyle=dashed,dash=1pt 1pt](30,14)(27,20)(27,41) 
        \psline[linestyle=dashed,dash=1pt 1pt](28,15)(20,19)(20,40) 
        \psline[linestyle=dashed,dash=1pt 1pt](30,14)(35,11.5) 
        \psline[linestyle=dashed,dash=1pt 1pt](28,15)(34,15)(34,37) 
        \psline[linestyle=dashed,dash=1pt 1pt](13,15)(7,15)(7,34)   
        \psline[linestyle=dashed,dash=1pt 1pt](13,15)(17,19)(17,40) 
        \bol
        \psline{<->}(21,23)(21,33)
    \end{pspicture}}}
    \subfloat[b][Quadratic Constraint.]{
    \parbox[c][45mm]{42mm}{\centering
    \begin{pspicture}(0,0)(42,45)
        \stl
        \psellipticarc(21,38)(9,6){224}{340} 
        \psellipticarc(27,36)(12,8){137}{195} 
        \psellipticarc(16,34)(13,9){5}{64} 
        \psellipticarc[linestyle=dashed,dash=1pt 1pt](21,13)(9,6){224}{340} 
        \psellipticarc[linestyle=dashed,dash=1pt 1pt](27,11)(12,8){137}{195} 
        \psellipticarc[linestyle=dashed,dash=1pt 1pt](16,9)(13,9){5}{64} 
        \psline(15.8,25)(15.8,33.1) 
        \psline(20.1,32)(20.1,42.2) 
        \psline(28.85,28.2)(28.85,35.15) 
        \psline[linestyle=dashed,dash=1pt 1pt](15.8,8.1)(15.8,25) 
        \psline[linestyle=dashed,dash=1pt 1pt](20.1,17.2)(20.1,32) 
        \psline[linestyle=dashed,dash=1pt 1pt](28.85,10.15)(28.85,28.2) 
        \thl
        \psellipticarc(21,38)(9,6){340}{224} 
        \psellipticarc(27,36)(12,8){195}{137} 
        \psellipticarc(16,34)(13,9){64}{5} 
        \psellipticarc[linestyle=dashed,dash=1pt 1pt](21,13)(9,6){340}{224} 
        \psellipticarc(27,11)(12,8){261}{0} 
        \psellipticarc[linestyle=dashed,dash=1pt 1pt](27,11)(12,8){195}{261} 
        \psellipticarc[linestyle=dashed,dash=1pt 1pt](27,11)(12,8){0}{137} 
        \psellipticarc(16,9)(13,9){180}{329} 
        \psellipticarc[linestyle=dashed,dash=1pt 1pt](16,9)(13,9){329}{5} 
        \psellipticarc[linestyle=dashed,dash=1pt 1pt](16,9)(13,9){64}{180} 
        \psline(3,9)(3,34) 
        \psline(39,11)(39,36) 
        \psline(25.7,3.1)(25.7,28.1) 
        \psline[linestyle=dashed,dash=1pt 1pt](16,17.9)(16,43) 
        \psline[linestyle=dashed,dash=1pt 1pt](23.8,18.6)(23.8,43.6) 
        \psline(12,25.5)(12,38) 
        \psline[linestyle=dashed,dash=1pt 1pt](12,13)(12,25.5) 
        \psline(30,28.2)(30,38) 
        \psline[linestyle=dashed,dash=1pt 1pt](30,13)(30,28.2) 
        \bol
        \psline{->}(22.5,32.1)(22.5,38)
        \psline[linestyle=dashed,dash=1pt 1pt]{<-}(22.5,28)(22.5,32.1)
    \end{pspicture}}}
    \caption{Illustration of Example \ref{Exa:SupportRank} in $\BR^{3}$. The arrows indicate the
    dimension of the \emph{unconstrained subspace}, equal to $3$ minus the respective \emph{support
    rank} $\alpha$, $\beta$, or $\gamma$. \label{Fig:SupportRank}}
\end{figure}

Let $\CL$ be the collection of all linear subspaces in $\BRd$.
In order to be unconstrained, select only those subspaces that are contained in almost all level
sets $F_{i}(x,\de)$:
\begin{equation}\label{Equ:SubSpace}
	\CL_{i}:=\bigcap_{\de\in\Delta}\bigcap_{x\in\BRd}
	\bigl\{L\in\CL\:\big|\:L\subset F_{i}(x,\de)\bigr\}\ef
\end{equation}
Introduce `$\preceq$' as the partial order on $\CL_{i}$ defined by set inclusion; \ie for any two
subspaces $L,L'\in\CL_{i}$, $L\preceq L'$ if and only if $L\subseteq L'$. Then the following
concepts are well-defined, as shown in Proposition \ref{The:UnconstrSub} below.

\vspace*{0.15cm}
\begin{definition}[Unconstrained Subspace, Support Rank]\label{Def:SupportRank}
    (a) The \emph{unconstrained subspace} $L_{i}$ of chance constraint $i\in\BN_{1}^{N}$ is the
    unique maximal element in $\CL_{i}$, in the sense that $L\preceq L_{i}$ for all $L\in\CL_{i}$.
    (b) The \emph{support rank} $\sri\in\BN_{0}^{d}$ of chance constraint $i\in\BN_{1}^{N}$ equals
    to $d$ minus the dimension of $L_{i}$,
	\begin{equation*}
		\sri:=d-\dim L_{i}\ef
	\end{equation*}
	\vspace*{-0.5cm}
\end{definition}
\vspace*{0.15cm}

It is a minor technicality in Definition \ref{Def:SupportRank} that any $\Pb$-null set that
adversely influences the dimension of the unconstrained subspace can be removed from $\Delta$; this
is tacitly understood.

Observe that if $\CL_{i}$ contains only the trivial subspace, then the support rank is actually
equal to Helly's dimension $d$.
On the other hand, if $\CL_{i}$ contains more than the trivial subspace, then the support rank
becomes strictly less than $d$.

\vspace*{0.15cm}
\begin{proposition}[Well-Definedness of Unconstrained Subspace]\label{The:UnconstrSub}
    The collection $\CL_{i}$ contains a unique maximal element $L_{i}$ in the set-inclusion sense,
    \ie $L_{i}$ contains all other elements of $\CL_{i}$ as subsets.
\end{proposition}
\vspace*{0.15cm}

\begin{proof}
    First, note that $\CL_{i}$ is always non-empty, because for every $x\in\BX$ and every $\de\in
    \Delta$ the level set $F_{i}(x,\de)$ includes the origin by its definition in
    \eqref{Equ:LevelSets}. Therefore $\CL_{i}$ contains (at least) the trivial subspace $\{0\}$.

    Second, since every chain in $\CL_{i}$ has an upper bound (namely $\BRd$), \emph{Zorn's Lemma}
    (or the \emph{Axiom of Choice}, \cf \cite[p.\,50]{Boll:1999}) implies that $\CL_{i}$ has at
    least one maximal element in the `$\preceq$'-sense.

    Third, in order to prove that the maximal element is unique, suppose that $L_{i}^{(1)},
    L_{i}^{(2)}$ are two maximal elements of $\CL_{i}$.
    It will be shown that their direct sum $L_{i}^{(1)}\oplus L_{i}^{(2)}\in\CL_{i}$, so that
    $L_{i}^{(1)}\neq L_{i}^{(2)}$ would contradict their maximality.
    According to \eqref{Equ:SubSpace}, it must be shown that $L_{i}^{(1)}\oplus L_{i}^{(2)}
    \subset F_{i}(x,\de)$ for any fixed values $x\in\BX$ and $\de\in\Delta$. To see this, pick
    \begin{equation*}
        \xi\in L_{i}^{(1)}\oplus L_{i}^{(2)}\quad\Longrightarrow\quad\xi=\xi^{(1)}+\xi^{(2)}
        \quad\text{for}\enspace\xi^{(1)}\in L_{i}^{(1)},\:\xi^{(2)}\in L_{i}^{(2)}\ef
    \end{equation*}
    Then apply \eqref{Equ:LevelSets} twice to obtain
    \begin{equation*}
        f_{i}(x+\xi^{(1)}+\xi^{(2)},\de)=f_{i}(x+\xi^{(1)},\de)=f_{i}(x,\de)\ec
    \end{equation*}
    because $\xi^{(2)}\in L_{i}^{(2)}$ and $\xi^{(1)}\in L_{i}^{(1)}$.
\end{proof}
\vspace*{0.15cm}

\subsection{The Support Rank Lemma}\label{Sec:RankLemma}

The following lemma provides the link between the support rank of a chance constraint and its
support dimension.

\vspace*{0.15cm}
\begin{lemma}[Support Rank]\label{The:RankLemma}
	Suppose that a chance constraint $i\in\BN_{1}^{N}$ has the support rank $\sri\in\BN_{1}^{d}$.
	Then its support dimension in the $\MSP$ is bounded by $\sdi\leq\sri$.
\end{lemma}
\vspace*{0.15cm}

\begin{proof}
    Without loss of generality, the proof is given for the first chance constraint $i=1$.
    Pick any random multi-sample $\bom\in\Delta^{K}$ (less any $\Pr^{K}$-null set for which the
    support rank condition may not hold).

    By the assumption, there exists a linear subspace $L_{1}\subset\BRd$ of dimension $d-\sr_{1}$
    for which
    \begin{equation*}
        f_{1}(x+\xi)=f_{1}(x)\qquad\fa x\in\BX,\:\:\fa\xi\in L_{1}\ef
    \end{equation*}
	The  orthogonal complement of $L_{1}$, $L_{1}^{\perp}$, is also a linear subspace of $\BRd$ with
	dimension $\sr_{1}$, and every vector in $\BRd$ can be uniquely written as the orthogonal sum of
	vectors in $L_{1}$ and $L_{1}^{\perp}$, \cf \cite[p.\,135]{Boll:1999}.
	
    For the sake of a contradiction, suppose that $i=1$ contributes more than $\sr_{1}$ support
    constraints to the resulting $\DSP$, \ie $|\bSc_{1}|\geq\sr_{1}+1$.
	For any $\kappa_{1}\in\bSc_{1}$, let
    \begin{equation*}
		\bxo_{\kappa_{1}}:=
		\bxo\bigl(\bom^{(1)}\setminus\{\bde^{(1,\kappa_{1})}\},\bom^{(2)},...,\bom^{(N)}\bigr)
    \end{equation*}
	be the solution obtained if this support constraint is omitted.
	By Definition \ref{Def:SupConstr}, if a support constraint is omitted from $\DSP$, its solution
	moves away from $\bxo_{0}$, \ie $\bxo_{0}\neq\bxo_{\kappa_{1}}$ for all $\kappa_{1}\in\bSc_{1}$.
	Denote the collection of all solutions by
	\begin{equation*}
		X:=\bigl\{\bxo_{\kappa_{1}}\:\big|\:\kappa_{1}\in\bSc_{1}\bigr\}\cup\{\bxo_{0}\}\ec
	\end{equation*}
	so that $|X|\geq\sr_{1}+2$.
	Observe that each $\bxo_{\kappa_{1}}$ is feasible with respect to all constraints of the $\DSP$,
	except for the one generated by $\de^{(1,\kappa_{1})}$, which is necessarily violated according
	to Definition \ref{Def:SupConstr}.
	
    Since $\BRd$ is the orthogonal direct sum of $L_{1}$ and $L_{1}^{\perp}$, for each point in $X$
    there is a unique orthogonal decomposition of
    \begin{equation*}
        \bxo_{\kappa_{1}}=v_{\kappa_{1}}+w_{\kappa_{1}}\ec\qquad\text{where}\enspace v_{\kappa_{1}}
        \in L_{1},\enspace w_{\kappa_{1}}\in L_{1}^{\perp}\ec
    \end{equation*}
    where $\kappa_{1}\in\bSc_{1}\cup\{0\}$.
    Consider the set
    \begin{equation*}
        W:=\bigl\{w_{\kappa_{1}}\:\big|\:\kappa_{1}\in\bSc_{1}\cup\{0\}\bigr\}\ef
    \end{equation*}
    By the hypothesis, $W$ contains at least $\sr_{1}+2$ distinct points in the
    $\sr_{1}$-dimensional subspace $L_{1}^{\perp}$.
    According to Radon's Theorem \cite[p.\,151]{Ziegler:2007}, $W$ can be split into two disjoint
    subsets, $W_{A}$ and $W_{B}$, such that there exists a point $\tilde{w}$ in the intersection of
    their convex hulls:
    \begin{equation}\label{Equ:WInConvHull1}
        \tilde{w}\in \conv\bigl\{W_{A}\bigr\}\cap \conv\bigl\{W_{B}\bigr\}\ef
    \end{equation}
    Split the indices in $\bSc_{1}\cup\{0\}$ correspondingly into $I_{A}$ and $I_{B}$, and observe
    that every $w_{A}\in W_{A}$ satisfies the constraints in $I_{B}$:
    \begin{equation*}
        f_{1}\bigl(w_{A},\bde^{(1,\kappa_{1})}\bigr)\leq 0\quad\fa\kappa_{1}\in I_{B}
        \qquad\Longrightarrow\qquad
        f_{1}\bigl(\tilde{w},\bde^{(1,\kappa_{1})}\bigr)\leq 0\quad\fa\kappa_{1}\in I_{B}\ef
    \end{equation*}
    The last implication follows because $\tilde{w}\in\conv\{W_{A}\}$ and $f_{1}(\,\cdot\,,
    \bde^{(1,\kappa_{1})})$ is convex.
    Similarly, every point $w_{B}\in W_{B}$ satisfies the constraints in $I_{A}$:
    \begin{equation*}
        f_{1}\bigl(w_{B},\bde^{(1,\kappa_{1})}\bigr)\leq 0\quad\fa\kappa_{1}\in I_{A}
        \qquad\Longrightarrow\qquad
        f_{1}\bigl(\tilde{w},\bde^{(1,\kappa_{1})}\bigr)\leq 0\quad\fa\kappa_{1}\in I_{A}\ef
    \end{equation*}
    Combining both statements thus yields
    \begin{equation}\label{Equ:ConstrSatisfaction}
        f_{1}(\tilde{w},\bde^{(1,\kappa_{1})})\leq 0\qquad\fa\kappa_{1}\in\bSc_{1}\ef
    \end{equation}

    According to \eqref{Equ:WInConvHull1}, $\tilde{w}$ can be expressed as a convex combination of
    elements in $W_{A}$ or $W_{B}$.
    Splitting the points in $X$ into $X_{A}$ and $X_{B}$ correspondingly and applying the same
    convex combination yields some
    \begin{equation}\label{Equ:XInConvHull}
        \tilde{x}\in\conv\bigl\{X_{A}\bigr\}\cap\conv\bigl\{X_{B}\bigr\}\ec
    \end{equation}
    and thereby also some $\tilde{v}\in L_{1}$ with $\tilde{x}=\tilde{v}+\tilde{w}$.
	
    To establish the contradiction two things remain to be verified: first that $\tilde{x}$ is
    feasible with respect to all constraints, and second that it has a lower cost (or a better
    tie-break value) than $\bxo_{0}$.
    For the first, $\tilde{x}\in\BX$ because all points of $X$ lie in $\BX$ and $\tilde{x}\in
    \conv\{X\}$. Moreover, thanks to \eqref{Equ:ConstrSatisfaction},
    \begin{equation*}
        f_{1}\bigl(\tilde{x},\bde^{(1,\kappa_{1})}\bigr)=f_{1}\bigl(\tilde{w},\bde^{(1,\kappa_{1})}
        \bigr)\leq 0\qquad\fa\kappa_{1}\in\bSc_{1}\ef
    \end{equation*}
    For the second, pick the set from $X_{A}$ and $X_{B}$ that does not contain $\bxo_{0}$; without
    loss of generality, say this is $X_{A}$.
    By construction, all elements of $X_{A}$ have a strictly lower objective function value (or at
    least a better tie-break value) than $\bxo_{0}$.
    By linearity this also holds for all points in $\conv\{X_{A}\}$, where $\tilde{x}$ lies
    according to \eqref{Equ:XInConvHull}.
\end{proof}

\vspace*{0.15cm}
\begin{remark}[Support Rank versus Support Dimension]\label{Rem:SuppDim}
    While the support rank $\sri$ is a property of chance constraint $i$ alone, the support
    dimension $\sdi$ may depend on the overall setup of the $\MSP$.
    The support dimension $\sdi$ constitutes the relevant basis for selecting the sample size
    $K_{i}$. However, it may be difficult to determine for practical problems, as it may depend on
    the interactions of multiple chance constraints (see Example \ref{Exa:SuppDim} below).
    The support rank $\sri$ provides an easier-to-handle upper bound to $\sdi$, which can be used in
    place of $\sdi$ for selecting $K_{i}$.
\end{remark}
\vspace*{0.15cm}

\begin{example}[Upper Bounding of Support Dimension]\label{Exa:SuppDim}
	\textnormal{To illustrate the statements in Remark \ref{Rem:SuppDim}, consider a small example
	of \eqref{Equ:MCP} in dimension $d=3$.
	Let $\BX=[-1,1]^{3}$ be the unit cube, $c\tp=[0\,1\,1]$ with a lexicographic tie-break rule, and
	two chance constraints $i=1,2$. Both constraints affect only the first and second coordinates
	$x_{1}$ and $x_{2}$, leaving the choice of $x_{3}=-1$ for the third coordinate.
	For $i=1$, the constraints are parallel hyperplanes constraining $x_{1}$ from below, where the
	lower bound is given by the first uncertainty $\de_{1}$:
	\begin{equation*}
		f_{1}(x,\de) = -x_{1}+\de_{1}\ef
	\end{equation*}
	For $i=2$, the constraints are V-shaped, with the vertex located at $x_{1}=-\de_{2}$
	and $x_{2}=-1$:
	\begin{equation*}
		f_{2}(x,\de) = \bigl|x_{1}+\de_{2}\bigr| - x_{2} - 1\ef
	\end{equation*}
	Both uncertainties $\de:=\{\de_{1},\de_{2}\}$ are uniformly distributed on the interval $[0,1]$.
	The setup is illustrated in Figure \ref{Fig:SuppDim}.}
	
	\textnormal{In this case, the support dimensions are $\sd_{1}=1$, $\sd_{2}=1$ and the support
	ranks are  $\sr_{1}=1$, $\sr_{2}=2$ for the constraints $i=1,2$. Notice that for $i=2$ the
	support rank is strictly greater that its support dimension, due to the presence of constraint
	$1$. Hence there is some conservatism in the upper bound, although both bounds are better than
	the existing upper bound by the dimension of the decision space $d=3$ \cite[Thm.\,2]{CalaCamp:2005}.}
\end{example}
\vspace*{0.15cm}

\begin{figure}[H]
    \centering
    \begin{pspicture}(-40,-40)(30,30)
            \stl
        \psline[arrowsize=4.5pt]{->}(-30,0)(30,0)
        \rput[bl](26,-5){$x_{1}$}
        \psline[arrowsize=4.5pt]{->}(0,-30)(0,30)
        \rput[bl](-5,26){$x_{2}$}
        \psline[linestyle=dashed,dash=1pt 1pt](-25,-25)(-25,25)(25,25)(25,-25)(-25,-25)
        \rput[bl](-29,23){$\BX$}
        \psline[linewidth=2pt]{->}(30,30)(30,22)
        \rput[bl](31,26.5){$-c$}
        \mel
        \psdots*[dotstyle=*,dotsize=2pt](4,0)
        \psline(4,-26)(4,26)
        \pscustom[linestyle=none,fillstyle=hlines,hatchwidth=0.2pt,hatchsep=3pt,hatchangle=-30]{
        	\pspolygon(1.18,-26)(4,-26)(4,26)(1.18,26)(1.18,-26)
			}
        \psdots*[dotstyle=*,dotsize=2pt](10,0)
        \psline(10,-26)(10,26)
        \pscustom[linestyle=none,fillstyle=hlines,hatchwidth=0.2pt,hatchsep=3pt,hatchangle=-30]{
        	\pspolygon(7.18,-26)(10,-26)(10,26)(7.18,26)(7.18,-26)
			}
	    \rput[bl](2.5,-30){$i=1$}
        \psdots*[dotstyle=*,dotsize=2pt](-19,-25)
        \psline(-28,-16)(-19,-25)
        \psline(-19,-25)(26,20)
        \pscustom[linestyle=none,fillstyle=hlines,hatchwidth=0.2pt,hatchsep=3pt,hatchangle=-70]{
        	\pspolygon(-28,-16)(-19,-25)(26,20)(28,18)(-19,-27.82)(-30,-18)(-28,-16)
			}
	    \psdots*[dotstyle=*,dotsize=2pt](-11,-25)
	    \psline(-26,-10)(-11,-25)
        \psline(-11,-25)(28,14)
        \pscustom[linestyle=none,fillstyle=hlines,hatchwidth=0.2pt,hatchsep=3pt,hatchangle=-70]{
        	\pspolygon(-26,-10)(-11,-25)(28,14)(30,12)(-11,-27.82)(-28,-12)(-26,-10)
			}
	    \rput[bl](-36,-15){$i=2$}
    \end{pspicture}
    \caption{Illustration of Example \ref{Exa:SuppDim}. The plot shows a projection on the $x_{1},
    x_{2}$-plane for $x_{3}=-1$. The unit box $\BX$ is depicted by a dotted line. Two (possible)
    samples are shown for the linear constraint $i=1$ ($x_{1}\geq\de_{1}$) and for the V-shaped
    constraint $i=2$ ($x_{2}\geq\bigl|x_{1}+\de_{2}\bigr|-1$).\label{Fig:SuppDim}}
\end{figure}
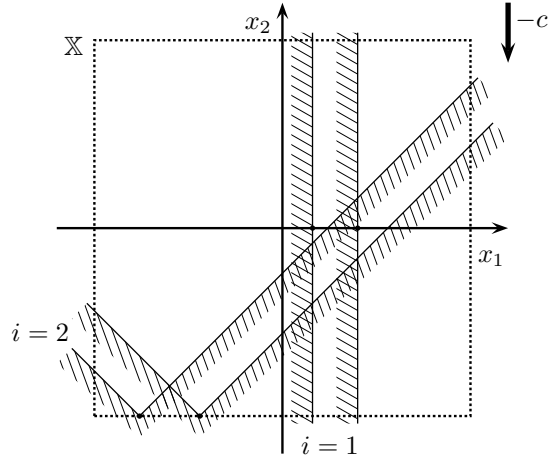

\section{Feasibility of the Scenario Solution}\label{Sec:ScenSol}

In the first part of this section, it is shown that for a proper choice of the sample sizes
$K_{1},...,K_{N}$ the scenario solution $\xo\bigl(\om^{(1)},...,\om^{(N)}\bigr)$ is an
approximate solution of the $\MCP$ (\ie it is a feasible point of each chance constraint $i=1,...,N$
in (\ref{Equ:MCP}b) with a high confidence $(1-\theta_{i})$).
In the second part of this section, an explicit formula for computing the sample sizes $K_{1},...,
K_{N}$ for given residual probabilities $\theta_{i}$ is provided.

\subsection{The Sampling Theorem}\label{Sec:MainThe}

Denote by $\B(\cdot\,;\cdot,\cdot)$ the beta distribution function, \cf \cite[p.\,26.5.3,\,26.5.7]{Abramowitz:1970}:
\begin{equation}\label{Equ:BinDistr}
    \B(\ep;n,K):=\sum_{j=0}^{n}{K\choose j}\ep^{j}(1-\ep)^{K-j}\ef
\end{equation}

\vspace*{0.15cm}
\begin{theorem}[Sampling Theorem]\label{The:Sampling}
    Consider problem \eqref{Equ:MSP} under Assumptions \ref{Ass:Convexity}, \ref{Ass:Uniqueness},
    \ref{Ass:Independence}, \ref{Ass:Feasibility}, \ref{Ass:SampleSize}. Then
    \begin{equation}\label{Equ:Sampling}
	    \Pb^{K}\bigl[\V_{i}(\om^{(1)},...,\om^{(N)})>\ep_{i}\bigr]\leq\B(\ep_{i};\sri-1,K_{i})\ec
    \end{equation}
    for each chance constraint $i\in\BN_{1}^{N}$, whose support rank is $\sri$.
\end{theorem}
\vspace*{0.15cm}

\begin{proof}
    The result is an extension of \cite[Thm.\,2.4]{CampGar:2008} for the classic scenario approach,
    which is also used as a basis for this proof.\footnote{The authors thank an anonymous reviewer
    for his/her helpful suggestions on simplifying the proof.}

    Without loss of generality, consider the first chance constraint $i=1$; the result for the other
    chance constraints $i=2,...,N$ follows analogously.
    Consider the conditional probability
    \begin{equation}\label{Equ:CondProb}
        \Pb^{K}\bigl[\V_{1}(\om^{(1)},...,\om^{(N)})>\ep_{1}\:\big|\:
        \om^{(2)},...,\om^{(N)}\bigr]\ec
    \end{equation}
    \ie the probability of drawing $\om^{(1)}$ such that $\xo(\om^{(1)},...,\om^{(N)})$ has a
    probability of violating `$f_{1}(\,\cdot\,,\den)\leq 0$' that is higher than $\ep_{1}$, given
    fixed values for the other samples $\om^{(2)},...,\om^{(N)}$.

    Clearly, the quantity in \eqref{Equ:CondProb} generally depends on the multi-samples $\om^{(2)},
    ...,\om^{(N)}$. However, for $\Pb^{K_{2}+...+K_{N}}$-almost every value of these multi-samples
    \eqref{Equ:CondProb} can be bounded by
    \begin{equation}\label{Equ:CondProbBound}
        \Pb^{K}\bigl[\V_{1}(\om^{(1)},...,\om^{(N)})>\ep_{1}\:\big|\:\om^{(2)},...,\om^{(N)}\bigr]
        \leq \B(\ep_{1};\sr_{1}-1,K_{1})\ef
    \end{equation}
    Indeed, by Assumption \ref{Ass:Convexity}, for $\Pb^{K_{2}+...+K_{N}}$-almost every
    $\om^{(2)},...,\om^{(N)}$ the function $\tilde{f}:\BRd\to\BR$ defined by
    \begin{equation*}
    	\tilde{f}(x)\equiv \max_{i\in\BN_{2}^{N}}\max_{\kappa_{i}\in\BN_{1}^{K_{i}}}
	    f_{i}\bigl(x,\de^{(i,\kappa_{i})}\bigr)
    \end{equation*}
    is convex, as it is the point-wise maximum of convex functions.
    Then all sampled constraints of $i=2,...,N$ can be expressed as the deterministic convex
    constraint `$\tilde{f}(x)\leq 0$', which can be considered as part of the convex set $\BX$.
    Thus for $\Pb^{K_{2}+...+K_{N}}$-almost every $\om^{(2)},...,\om^{(N)}$ the problem takes the
    form of a classic $\SCP$, to which the results of \cite{CampGar:2008} apply.
    In particular, \cite[Thm.\,2.4]{CampGar:2008} yields \eqref{Equ:CondProbBound} for
    $\Pb^{K_{2}+...+K_{N}}$-almost every $\om^{(2)},...,\om^{(N)}$.

    The difference from using the support rank $\sr_{1}$ in place of the optimization dimension $d$
    in \cite[Thm.\,2.4]{CampGar:2008} is minor.
    The key fact is that $\sr_{1}$ provides an upper bound for the number of support constraints
    contributed by constraint $1$, according to Lemma \ref{The:RankLemma}, and hence it can replace
    $d$ in \cite[Prop.\,2.2]{CampGar:2008} and all subsequent results.

    The final result is obtained by deconditioning the probability in \eqref{Equ:CondProb}:
    \begin{subequations}\begin{align*}
        \Pb^{K}&\bigl[\V_{1}(\om^{(1)},...,\om^{(N)})>\ep_{1}\bigr]=\\
        &=\int_{\om^{(2)},...\om^{(N)}}
        \Pb^{K}\bigl[\V_{1}(\om^{(1)},...,\om^{(N)})>\ep_{1}\:\big|\:\om^{(2)},...,\om^{(N)}\bigr]
        \Pb^{K_{2}}\bigl[\di\om^{(2)}\bigr]...\Pb^{K_{N}}\bigl[\di\om^{(N)}\bigr]\\
        &\leq\int_{\om^{(2)},...\om^{(N)}}
        \Phi(\ep_{1};\sr_{1}-1,K_{1})
        \Pb^{K_{2}}\bigl[\di\om^{(2)}\bigr]...\Pb^{K_{N}}\bigl[\di\om^{(N)}\bigr]\\
        &=\Phi(\ep_{1};\sr_{1}-1,K_{1})\ec
    \end{align*}\end{subequations}
    based on \cite[pp.\,183,222]{Shir:1996}, where the third line uses \eqref{Equ:CondProbBound}.
\end{proof}

\subsection{Explicit Bounds on the Sample Sizes}\label{Sec:Chernoff}

Formula \eqref{Equ:Sampling} in Theorem \ref{The:Sampling} ensures that with a
\emph{confidence level} of $1-\B(\ep_{i};\sr_{i}-1,K_{i})$, the violation probability
$\V_{i}(\om^{(1)},...,\om^{(N)})\leq\ep_{i}$.
However, in practical applications a given confidence level $(1-\theta_{i})\in (0,1)$ is often
imposed, while an appropriate sample size $K_{i}$ has to be identified.

The most accurate way of finding this sample size is by observing that
$\B(\ep_{i};\sr_{i}-1,K_{i})$ is a monotonically decreasing function in $K_{i}$ and applying a
numerical procedure (\eg regula falsi) for computing the smallest sample size that ensures
$\B(\ep_{i};\sr_{i}-1,K_{i})\leq\theta_{i}$.
The resulting $K_{i}$ shall be referred to as the \emph{implicit bound} on the sample size.

For a qualitative analysis of the behavior of this implicit bound as $\ep_{i}$ and $\theta_{i}$ vary
(and also for a good initialization of the regula falsi procedure), it is useful to derive an
\emph{explicit bound} on the sample size $K_{i}$.
Since formula \eqref{Equ:Sampling} cannot be readily inverted, the beta distribution function
must first be controlled by some upper bound, which is then inverted.

A straightforward approach is to use a Chernoff bound \cite{Cher:1952}, as shown in \cite[Rem.\,2.3]{Cala:2009} and \cite[Sec.\,5]{Cala:2010}. This provides a simple explicit formula for $K_{i}$:
\begin{equation}\label{Equ:ExpBound1}
	K_{i}\geq\displaystyle\frac{2}{\ep_{i}}
	\left[\log\Bigl(\frac{1}{\theta_{i}}\Bigr)+\sr_{i}-1\right]\ec
\end{equation}
where $\log(\cdot)$ denotes the natural logarithm.
As shown in \cite[Cor.\,1]{AlamoEtAl:2010}, this can be further improved to a better, albeit more
complicated bound for $K_{i}$:
\begin{equation}\label{Equ:ExpBound2}
	K_{i}\geq\displaystyle\frac{1}{\ep_{i}}
	\left[\log\Bigl(\frac{1}{\theta_{i}}\Bigr)+
	\sqrt{2(\sr_{i}-1)\log\Bigl(\frac{1}{\theta_{i}}\Bigr)}+\sr_{i}-1\right]\ef
\end{equation}

\section{The Sampling-and-Discarding Approach}\label{Sec:DiscardConstr}

The sampling-and-discarding approach has previously been proposed for the classic scenario approach
\cite{Cala:2010,CampGar:2011}; this section describes its extension to problems with multiple chance
constraints.

The fundamental goal is to reduce the objective value of the scenario solution, while maintaining
the same confidence levels for feasibility with respect to the chance constraints (see Section
\ref{Sec:SCP}).
To this end, the sample sizes $K_{i}$ are deliberately increased above the bounds derived in
Section \ref{Sec:ScenSol}, in exchange for allowing a certain number of $R_{i}$ sampled constraints
to be discarded \emph{a posteriori}, \ie after the outcomes of the samples have been observed.

In this section, first the possible procedures for discarding constraints are recalled. Second, the
main result on the sampling-and-discarding approach for the $\MCP$ is stated. It provides an
implicit formula for the selection of appropriate sample-and-discarding pairs $(K_{i},R_{i})$, which
may again vary for different chance constraints $i=1,...,N$. Third, explicit bounds for the choice
of pairs $(K_{i},R_{i})$ are provided.

\subsection{Constraint Discarding Procedure}\label{Sec:DiscardProc}

For each chance constraint of the $\MCP$, if $R_{i}\geq 0$ sampled constraints are to be discarded a
posteriori, the discarding procedure is performed by a pre-defined \emph{(sample) removal
algorithm}.

\vspace*{0.15cm}
\begin{definition}[Removal Algorithm]\label{Def:RemAlg}
    For each chance constraint $i=1,...,N$, the \emph{(sample) removal algorithm}
    $\CA_{i}^{(K_{i},R_{i})}:\Om\to\tOmi$ is a deterministic function on the overall multi-sample
    $\om\in\Om$.
    It returns a subset of samples $\tom^{(i)}\in\tOmi$, in which $R_{i}$ out of the $K_{i}$ samples
    in $\om^{(i)}\in\Omi$ have been removed.
\end{definition}
\vspace*{0.15cm}

Obviously, the algorithm should aim at improving the objective value from
$\MSP[\om^{(1)},...,\om^{(N)}]$ to $\MSP[\tom^{(1)},...,\tom^{(N)}]$ as much as possible.
Various possible removal algorithms are described in \cite[Sec.\,5.1]{Cala:2010}, and further
references are found in \cite[Sec.\,2]{CampGar:2011}.
Brief descriptions of the most important removal algorithms are listed below.

\vspace*{0.15cm}
\begin{example}\label{Exa:ConstrRem}
    (a) Optimal Constraint Removal. \textnormal{The best improvement of the objective function value
    is achieved by solving the reduced problem for all possible ways of removing $R_{i}$ of the
    $K_{i}$ samples.
    However, a major drawback of this removal algorithm is its combinatorial complexity. Therefore
    the algorithm becomes computationally intractable for larger values of $R_{i}$, in particular
    when samples have to be removed for multiple constraints.}

    (b) Greedy Constraint Removal. \textnormal{Starting by solving the $\MSP[\om^{(1)},...,
    \om^{(N)}]$ for all $K_{i}$ samples, the $R_{i}$ samples are removed in $R_{i}$ sequentially
    steps. In each step, a single sample is removed by the optimal constraint removal procedure.
    Between multiple constraints $i$, the removal algorithm can either proceed in a fixed order or
    again greedy-based. For most practical problems this algorithm can be expected to work almost
    as good as (a), while carrying a much lower computational burden.}

    (c) Marginal Constraint Removal. \textnormal{The $R_{i}$ samples are removed in $R_{i}$
    sequential steps, where the removed sample in each step is selected according to the highest
    Lagrange multiplier.
    Compared to the greedy constraint removal, the decision is thus based on the highest marginal
    cost improvement \cite[Cha.\,5]{BoydVan:2004}), instead of the highest total cost improvement.
    In the case of multiple constraints $i$, the removal algorithm can either handle them all
    together, or proceed sequentially.}
\end{example}
\vspace*{0.15cm}

The existing theory for the $\SCP$ \cite[Sec.\,4.1.1]{Cala:2010} and \cite[Ass.\,2.2]{CampGar:2011}
assumes that all of the removed constraints are violated by the relaxed scenario solution.

\vspace*{0.15cm}
\begin{assumption}[Violation of Discarded Constraints]\label{Ass:DiscViol}
    Every chance constraint $i\in\BN_{1}^{N}$ with $R_{i}>0$ satisfies the following condition: for
    almost every $\om\in\Om$, each of the constraints discarded by the removal algorithm
    $\CA_{i}^{(K_{i},R_{i})}(\om)$ is violated by the solution of the reduced problem, \ie
    \begin{equation}\label{Equ:DiscViol}
         f_{i}\bigl(\xo(\tom^{(1)},...,\tom^{(N)}),\de^{(i,\kappa_{i})}\bigr)>0\qquad
         \fa\de^{(i,\kappa_{i})}\in\bigl(\om\setminus\tom\bigr)\ef
    \end{equation}
\end{assumption}
\vspace*{-0.2cm}

While Assumption \ref{Ass:DiscViol} is sufficient for the $\MCP$ as well, it may turn out to be
too restrictive for some problem instances.
In fact, due to the interplay of multiple chance constraints, it may not be possible to find $R_{i}$
constraints that are violated by the relaxed scenario solution (this situation may also occur for a
single chance constraint, in the presence of a deterministic constraint set $\BX$).
In this case, the \emph{monotonicity property}, as introduced below, provides a possible
alternative.

\vspace*{0.15cm}
\begin{definition}[Monotonicity Property]\label{Def:Monotonicity}
    A chance constraint $i\in\BN_{1}^{N}$ is called \emph{monotonic} if for all $K_{i}\in\BN$ and
    almost every $\om^{(i)}\in\Omi$ the following condition holds:
    Every point in the feasible set of sampled instances of chance constraint $i$,
    \begin{equation}\label{Equ:DefMono1}
        \BX_{i}(\om^{(i)}):=\bigl\{\xi\in\bBR^{d}\:\big|\:f_{i}(\xi,\de^{(i,\kappa_{i})})\leq 0
        \:\:\:\fa\kappa_{i}\in\BN_{1}^{K_{i}}\bigr\}\ec
    \end{equation}
    where $\bBR:=\BR\cup\{\pm\infty\}$, is violated by a new sampled constraint only if also the
    optimal point in $\BX_{i}(\om^{(i)})$,
    \begin{equation}\label{Equ:DefMono2}
        \xo_{i}(\om^{(i)}):=\argmin\bigl\{c\tp\xi\:\big|\:\xi\in\BX_{i}(\om^{(i)})\bigr\}
    \end{equation}
    is violated.
    In other words, for every $\xi\in\BX_{i}(\om^{(i)})$ and almost every $\den\in\Delta$,
    \begin{equation}\label{Equ:DefMono3}
        f_{i}\bigl(\xi,\den\bigr)>0\qquad\Longrightarrow\qquad
        f_{i}\bigl(\xo_{i}(\om^{(i)}),\den\bigr)>0\ef
    \end{equation}
\end{definition}
\vspace*{-0.3cm}

\begin{assumption}[Monotonicity of Chance Constraints]\label{Ass:Monotonicity}
    Every chance constraint $i\in\BN_{1}^{N}$ enjoys the \emph{monotonicity property}.
\end{assumption}
\vspace*{0.15cm}

Definition \eqref{Def:Monotonicity} is easy to check for most practical problems, without involving
any calculations. The following example illustrates the intuition behind this concept.

\vspace*{0.15cm}
\begin{example}[Monotonic Chance Constraints]\label{Exa:Monotonicity}
    \textnormal{Consider an $\MSP$ in $d=2$ dimensions, where $\BX=[-100,100]^{2}\subset\BR^{2}$ and
    $c=[\:0\:\:1\:]\tp$, $\den=[\den_{1}\:\den_{2}\:\den_{3}]$ belongs to $\Delta=\{-1,1\}\times
    [-1,1]\times [-1,1]$, and there are $N=2$ chance constraints.}

    (a) Monotonic Chance Constraint. \textnormal{Let the first chance constraint $i=1$ be of the
    linear form}
    \begin{equation*}
        \begin{bmatrix}
	       \de^{(1,\kappa_{1})}_{1} &1
	    \end{bmatrix}
	    x-\de^{(1,\kappa_{1})}_{2}\leq 0\qquad
        \fa\kappa_{1}=1,...,K_{1}\ef
    \end{equation*}
    \textnormal{Observe that for any number $K_{1}\in\BN$ and every possible sample values
    $\om^{(1)}$, an additional sample $\den$ either cuts off no point from $\BX_{1}(\om^{(1)})$,
    or the the point $\xo_{1}(\om^{(1)})$ becomes infeasible.
    This fact is illustrated in Figure \ref{Fig:Monotonicity}(a).
    Therefore chance constraint $i=1$ enjoys the monotonicity property.}

    (b) Non-Monotonic Chance Constraint. \textnormal{Let the second chance constraint $i=2$ be of
    the linear form}
     \begin{equation*}
     	\begin{bmatrix}
	       \de^{(2,\kappa_{2})}_{2} &1
	    \end{bmatrix}
	    x-\de^{(2,\kappa_{2})}_{3}\leq 0\qquad
        \fa\kappa_{2}=1,...,K_{2}\ec
    \end{equation*}
    \textnormal{Observe that for any number $K_{2}$ there exist sample values $\om^{(2)}$ that make
    it possible for a new sample $\den$ to cut off some previously feasible point from
    $\BX_{2}(\om^{(2)})$, without rendering the point $\xo_{2}(\om^{(2)})$ infeasible.
    A possible configuration of this type is depicted in Figure \ref{Fig:Monotonicity}(b).
    Therefore chance constraint $i=2$ does not enjoy the monotonicity property.}
\end{example}
\vspace*{0.15cm}

\begin{figure}[H]
    \centering
    \subfloat[b][Monotonic Chance Constraint.]{
    \parbox[c][50mm]{55mm}{\centering
    \begin{pspicture}(0,0)(50,55)
        \psline[linewidth=2pt]{->}(15,50)(15,42)
        \rput[bl](16,45.5){$-c$}
        \mel
        \psline(2,32)(32,2)
        \pscustom[linestyle=none,fillstyle=hlines,hatchwidth=0.2pt,hatchsep=3pt,hatchangle=30]{
        	\pspolygon(2,32)(32,2)(30,0)(0,30)
			}
	    \psline(2,42)(34,10)
        \pscustom[linestyle=none,fillstyle=hlines,hatchwidth=0.2pt,hatchsep=3pt,hatchangle=30]{
        	\pspolygon(2,42)(34,10)(32,8)(0,40)
			}
		\psline(48,45)(16,11)
        \pscustom[linestyle=none,fillstyle=hlines,hatchwidth=0.2pt,hatchsep=3pt,hatchangle=-30]{
        	\pspolygon(48,45)(16,11)(18,9)(50,43)
			}
	    \psline(48,32)(19,3)
        \pscustom[linestyle=none,fillstyle=hlines,hatchwidth=0.2pt,hatchsep=3pt,hatchangle=-30]{
        	\pspolygon(48,32)(19,3)(21,1)(50,30)
			}	
	    \bol
        \psline(2,48)(36.5,13.5)
        \pscustom[linestyle=none,fillstyle=hlines,hatchwidth=0.2pt,hatchsep=3pt,hatchangle=30]{
        	\pspolygon(2,48)(36.5,13.5)(34.5,11.5)(0,46)
			}
    \end{pspicture}}}
    \subfloat[b][Non-Monotonic Chance Constraint.]{
    \parbox[c][50mm]{55mm}{\centering
    \begin{pspicture}(0,3)(50,50)
        \psline[linewidth=2pt]{->}(15,50)(15,42)
        \rput[bl](16,45.5){$-c$}
        \mel
		\psline(48,48)(16,14)
        \pscustom[linestyle=none,fillstyle=hlines,hatchwidth=0.2pt,hatchsep=3pt,hatchangle=-30]{
        	\pspolygon(48,48)(16,14)(18,12)(50,46)
			}
	    \psline(2,35.5)(45,25)
        \pscustom[linestyle=none,fillstyle=hlines,hatchwidth=0.2pt,hatchsep=3pt,hatchangle=-60]{
        	\pspolygon(2,35.5)(45,25)(44.33,22.25)(1.33,32.75)
			}
		\psline(3,25.4)(40,2.4)
        \pscustom[linestyle=none,fillstyle=hlines,hatchwidth=0.2pt,hatchsep=3pt,hatchangle=30]{
        	\pspolygon(3,25.4)(40,2.4)(38.5,0)(1.5,23)
			}
		\psline(4,5)(46,15)
        \pscustom[linestyle=none,fillstyle=hlines,hatchwidth=0.2pt,hatchsep=3pt,hatchangle=-60]{
        	\pspolygon(4,5)(46,15)(46.65,12.25)(4.65,2.25)
			}
	    \bol
        \psline(2,45)(33,14)
        \pscustom[linestyle=none,fillstyle=hlines,hatchwidth=0.2pt,hatchsep=3pt,hatchangle=30]{
        	\pspolygon(2,45)(33,14)(31,12)(0,43)
			}
    \end{pspicture}}}
    \caption{Illustration of Example \ref{Exa:Monotonicity}. Non-bold constraints are generated by
    the multi-sample $\om^{(i)}\in\Delta^{K_{i}}$ of chance constraint $i=1,2$; bold constraints are
    generated by the uncertainty $\den\in\Delta$.
    In (b) a feasible point is made infeasible without affecting the optimum, which is not possible
    in the case of (a).\label{Fig:Monotonicity}}
\end{figure}
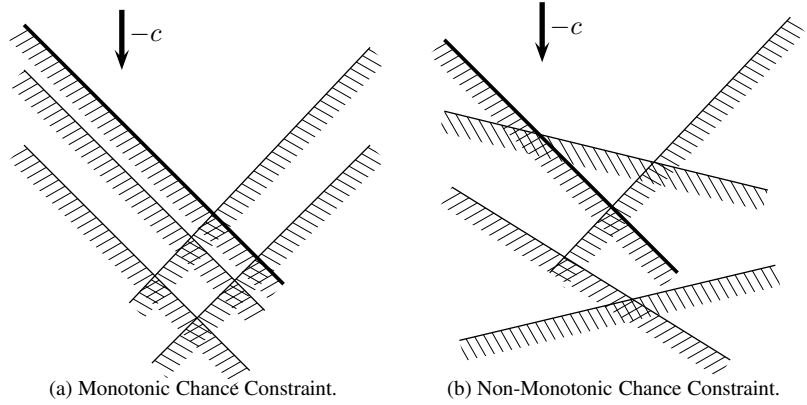

The usefulness of the monotonicity property is based on the following result, whose proof is an
straightforward consequence of Definition \ref{Def:Monotonicity} and therefore omitted.

\vspace*{0.15cm}
\begin{lemma}\label{The:Monotonicity}
    Let $K_{i}\in\BN$ and $R_{i}\leq K_{i}$.
    Suppose chance constraint $i\in\BN_{1}^{N}$ of $\MCP$ is monotonic and the removal algorithm
    $\CA_{i}^{(K_{i},R_{i})}$ is sequential. Then for almost every $\om^{(i)}\in\Delta^{K_{i}}$ the
    following holds:\\
    (a) With probability one every point $\xi$ in the set $\BX_{i}(\om^{(i)})$ has a violation
    probability less than or equal to that of the cost-minimal point $\xo_{i}(\om^{(i)})$:
    \begin{equation}\label{Equ:Monotonicity}
        \Pb\bigl[f_{i}(\xi,\den)>0\bigr]\leq
        \Pb\bigl[f_{i}(\xo_{i}(\om^{(i)}),\den)>0\bigr]
        \qquad\fa\xi\in\BX_{i}(\om^{(i)})\ef
    \end{equation}
    (b) The final solution $\xo_{i}(\tom^{(i)})$, where $\tom^{(i)}=
    \CA_{i}^{(K_{i},R_{i})}(\om_{i})$, violates all $R_{i}$ removed constraints.
\end{lemma}

\subsection{The Discarding Theorem}\label{Sec:DiscTheorem} 

For the sampling-and-discarding approach, the following result holds for the $\MCP$.

\vspace*{0.15cm}
\begin{theorem}[Discarding Theorem]\label{The:Discarding}
    Consider the problem \eqref{Equ:MCP} under Assumptions \ref{Ass:Convexity},
    \ref{Ass:Uniqueness}, \ref{Ass:Independence}, \ref{Ass:Feasibility}, \ref{Ass:SampleSize}, and
    either \ref{Ass:DiscViol} or \ref{Ass:Monotonicity}.
    Let $\CA_{i}^{(K_{i},R_{i})}$ be sample removal algorithms for each of its chance constraints
    $i=1,...,N$, some of which may be trivial (\ie $R_{i}=0$).
    Then it holds that
    \begin{equation}\label{Equ:Discarding}
	    \Pb^{K}\bigl[\V_{i}(\tom^{(1)},...,\tom^{(N)})>\ep_{i}\bigr]\leq
	    \displaystyle{R_{i}+\sri-1 \choose R_{i}}\B(\ep_{i};R_{i}+\sri-1,K_{i})\ec
    \end{equation}
    where $\sri$ denotes the support rank of chance constraint $i$ and $\B(\cdot;\cdot,\cdot)$ the
    beta distribution \eqref{Equ:BinDistr}.
\end{theorem}
\vspace*{0.15cm}

\begin{proof}
    Here the $\MCP$ case is reduced to the $\SCP$ case, for which a detailed proof is available
    in \cite[Sec.\,5.1]{CampGar:2011}.

    First, suppose that Assumption \ref{Ass:DiscViol} holds. The proof in \cite[Sec.\,5.1]{CampGar:2011} works analogously for an arbitrary chance constraint $i\in\BN_{1}^{N}$, given
    that an upper bound of the violation distribution is readily available from Theorem
    \ref{The:Sampling}.

    Second, suppose that Assumption \ref{Ass:Monotonicity} holds.
    In this case the proof in \cite[Sec.\,5.1]{CampGar:2011} can be applied directly to the $\SCP$
    which arises from the $\MCP$ if all chance constraints other than a particular
    $i\in\BN_{1}^{N}$ are omitted (and also $\BX$ is omitted).
    In particular, \eqref{Equ:Discarding} holds for the scenario solution of this $\SCP$, using
    Lemma \ref{The:Monotonicity}(b).
    Given that the chance constraint is monotonic and by virtue of Lemma \ref{The:Monotonicity}(a),
    \eqref{Equ:Discarding} also holds for any point in $\BX_{i}(\om^{(i)})$, in particular for the
    scenario solution of the $\MCP$.
\end{proof}
\vspace*{0.15cm}

The work of \cite{CampGar:2011} already provides an excellent account of the merits of the
sampling-and-discarding approach, which does not require a restatement here.
However, it should be emphasized that the scenario solution converges to the true solution
of the $\MCP$ as the number of discarded constraints increases, provided that the constraints are
removed by the optimal procedure of Example \ref{Exa:ConstrRem}(a).

\subsection{Explicit Bounds on the Sample-and-Discarding Pairs}\label{Sec:ExpDiscBounds} 

Similar to Section \ref{Sec:ScenSol}, explicit bounds on the sample size $K_{i}$ can also be derived
for the sampling-and-discarding approach, assuming the number of discarded constraints $R_{i}$ to be
fixed.
The technical details, using Chernoff bounds \cite{Cher:1952}, are worked out in \cite[Sec.\,5]{Cala:2010}.
The resulting explicit bound is indicated here for the sake of completeness,
\begin{equation}\label{Equ:ExpBoundSamp}
	K_{i}\geq\displaystyle\frac{2}{\ep_{i}}\log\biggl(\frac{1}{\theta_{i}}\biggr)+
	\frac{4}{\ep_{i}}\bigl(R_{i}+\sri-1\bigr)\ec
\end{equation}
where $\log(\cdot)$ denotes the natural logarithm.

Similarly, explicit bounds on the number of discarded constraints $R_{i}$ can be obtained, assuming
the sample size $K_{i}$ to be fixed:
\begin{equation}\label{Equ:ExpBoundDisc}
	R_{i}\leq\ep_{i}K_{i}-\sri+1-\displaystyle\sqrt{2\ep_{i}K_{i}
	\log\Bigl(\frac{(\ep_{i}K_{i})^{\sri-1}}{\theta_{i}}\Bigr)}\ef
\end{equation}
The technical details of this are found in \cite[Sec.\,4.3]{CampGar:2011}.

\section{Example: Minimal Diameter Cuboid}\label{Sec:Example}

The following academic example has been selected to highlight the strengths of the extensions to the
scenario approach presented in this paper.

\subsection{Problem Statement}

Let $\den$ be a random point in $\Delta\subset\BRn$, whose distribution and support set are unknown,
but sampled values can be obtained.
The objective in this example is to construct the Cartesian product $C$ of closed intervals in
$\BRn$ (`$n$-cuboid') of minimal $n$-diameter $W$, which is large enough to contain the point $\den$
in its $i$-th coordinate with probability $(1-\ep_{i})$. The setting is illustrated in Figure
\ref{Fig:ExampleRP}.

Let $z\in\BRn$ denote the center point of the cuboid and $t\in\BRn_{+}$ the interval widths in
each dimension, so that
\begin{equation}\label{Equ:CuboidDef}
	C=\bigl\{\xi\in\BRn\:\big|\:|\xi_{i}-z_{i}|\leq t_{i}/2\bigr\}\ef
\end{equation}
Then the corresponding stochastic program reads as follows:
\begin{subequations}\label{Equ:ExaNonConvProg}\begin{align}
    \min_{z\in\BRn,t\in\BRn_{+}}\quad&\|t\|_{2}\ec\\
    \st\quad&\Pr\bigl[z_{i}-t_{i}/2\leq\den_{i}\leq z_{i}+t_{i}/2\bigr]\geq(1-\ep_{i})
    \qquad\fa i\in\BN_{1}^{n}\ef
\end{align}\end{subequations}
Since the objective function is not linear, \eqref{Equ:ExaNonConvProg} has to be reformulated
(see Remark \ref{Rem:Generality}(a)) as
\begin{subequations}\label{Equ:ExaConvProg}\begin{align}
    &\min_{z\in\BRn,t\in\BRn_{+},T\in\BR}\:T\ec\hspace{8.9cm}\\
    &\quad\st\quad\|t\|_{2}\leq T\ec\\
    &\quad\pst\quad\Pr\Bigl[\max\bigl\{z_{i}-t_{i}/2-\den_{i},-z_{i}-t_{i}/2+\den_{i}\bigr\}\leq 0
    \Bigr]\geq(1-\ep_{i})\quad\fa i\in\BN_{1}^{n}\:.
\end{align}\end{subequations}

Note that \eqref{Equ:ExaConvProg} takes the form of a $\MCP$, for a $d=2n+1$ dimensional search
space and $N=n$ chance constraints: the objective function (\ref{Equ:ExaConvProg}a) is linear;
constraint (\ref{Equ:ExaConvProg}b) is deterministic and convex; and each of the chance constraints
in (\ref{Equ:ExaConvProg}c) is convex in $z,t$ for any fixed value of the uncertainty $\den\in
\Delta$.

Here each of the chance constraints $i=1,...,n$ depends on exactly two decision variables $z_{i}$
and $t_{i}$, which is a special case of involving $[z;t;T]\in\BR^{2n+1}$ (see Remark
\ref{Rem:Generality}(c)).
The convex and compact set $\BX$ is constructed from the positivity constraints on $t$, the
deterministic and convex constraint (\ref{Equ:ExaConvProg}b), and some artificial bounds assumed on
all variables.
Existence of a feasible solution, and hence Assumption \ref{Ass:Uniqueness}, holds automatically
from the problem setup.

\begin{figure}[t]
    \centering
    \begin{pspicture}(-50,-40)(50,40)
        \pspolygon[linestyle=none,fillstyle=solid,fillcolor=lightgray](-24.00,-10.97)(-24.00,22.00)
        (31.01,22.00)(31.01,-10.97)(-24.00,-10.97)
        \stl
        \psline[arrowsize=4.5pt]{->}(-48,0)(48,0)
        \rput[bl](44,-6){$\den_{1}$}
        \psline[arrowsize=4.5pt]{->}(0,-38)(0,38)
        \rput[bl](-6.5,34){$\den_{2}$}
        \psdots*[dotstyle=*,dotsize=3pt](9.38,-0.81)(11.06,-0.99)(16.43,13.27)(25.05,2.96)
        (5.07,8.88)(-11.35,5.81)(19.33,13.66)(-2.38,8.81)(21.91,-2.37)(1.41,2.30)
        (24.68,8.63)(-1.01,5.50)(-4.44,15.76)(18.70,3.75)(-7.24,-4.06)(15.61,6.31)
        (12.90,-0.73)(6.40,4.26)(4.18,9.60)(31.01,5.60)(3.81,5.56)(5.87,2.98)(-2.50,3.24)
        (-1.10,3.46)(1.47,-10.53)(27.12,6.04)(3.63,9.14)(0.28,-2.74)(2.99,-3.90)
        (15.99,7.22)(-6.71,-3.83)(15.01,-4.72)(-6.89,8.10)(6.54,0.60)(-12.43,10.46)
        (13.99,-2.71)(15.75,0.81)(0.31,5.26)(11.55,6.85)(9.82,6.43)(20.96,4.99)(9.44,2.81)
        (19.62,0.99)(6.69,6.50)(0.75,1.99)(7.58,6.19)(5.77,9.65)(17.86,-7.68)(7.37,13.37)
        (-1.41,0.72)(6.27,10.84)(21.29,12.06)(3.36,-3.63)(0.63,4.92)(19.00,22.00)(9.13,2.89)
        (3.94,6.53)(-0.26,5.72)(-8.55,10.58)(9.72,9.66)(4.45,-0.91)(11.57,1.72)(-8.29,5.78)
        (22.78,3.70)(13.45,17.16)(-9.52,6.28)(19.82,6.27)(1.43,-0.35)(11.38,-2.79)(10.45,8.12)
        (-1.82,2.98)(-24.00,7.37)(13.99,-0.32)(4.17,8.65)(1.31,-5.42)(12.68,12.14)(0.29,8.40)
        (15.51,4.57)(9.69,-2.03)(-3.09,1.94)(17.36,5.32)(13.31,10.27)(22.00,-1.37)(-1.98,10.73)
        (-2.01,18.37)(11.50,1.43)(15.46,6.43)(17.26,0.59)(8.07,-4.92)(5.36,11.64)(15.30,7.60)
        (29.97,1.35)(17.89,4.59)(8.95,2.81)(1.58,-3.59)(13.59,-10.97)(2.59,3.75)(10.67,1.60)
        (16.70,1.01)(18.34,-0.39)(14.58,-5.44)(23.41,1.29)(28.70,4.84)(-13.22,-1.51)(-11.89,6.34)
        (-2.32,2.58)(-14.34,3.47)(4.01,-3.42)(4.24,0.03)(10.80,6.13)(11.20,-0.10)(0.33,4.84)
        (-1.10,14.27)(16.87,8.13)(3.73,9.39)(-7.33,16.26)(8.12,10.10)(10.28,6.07)(3.40,14.13)
        (15.42,-0.70)(10.52,0.46)(5.62,14.48)(7.51,6.36)(13.12,3.78)(3.06,6.38)(-3.80,2.63)
        (6.26,2.48)(9.44,1.99)(13.70,1.12)(7.66,-0.90)(2.89,3.69)(14.67,1.02)(6.73,11.49)
        (17.03,5.22)(7.00,9.30)(7.93,3.27)(6.97,3.58)(-2.30,10.53)
        \thl
        \psline[linewidth=1.2pt,linestyle=dashed,dash=2pt 2pt](31.01,-35)(31.01,35)
        \psline[linewidth=1.2pt,linestyle=dashed,dash=2pt 2pt](-24.00,-35)(-24.00,35)
        \psline[linewidth=1.2pt,linestyle=dashed,dash=2pt 2pt](-45,22.00)(45,22.00)
        \psline[linewidth=1.2pt,linestyle=dashed,dash=2pt 2pt](-45,-10.97)(45,-10.97)
        \psline[arrowsize=5pt]{<->}(-24.00,-10.97)(31.01,22.00)
        \rput[bl](25,15){$T$}
    \end{pspicture}
    \caption{Illustration of the numerical example for $n=2$. The point $\den\in\Delta$ appears
    at random in $\BR^{2}$, according to some unknown distribution; the points drawn here are $166$
    \iid samples of $\den$.
    The objective is to construct the smallest product of two closed intervals (`2-cuboid'), drawn
    here as the shaded rectangle, such that the probability of failing to contain the realization of
    $\den$ is smaller than $\ep_{1}$ and $\ep_{2}$ in dimension $1$ and $2$, respectively.
    \label{Fig:ExampleRP}}
\end{figure}
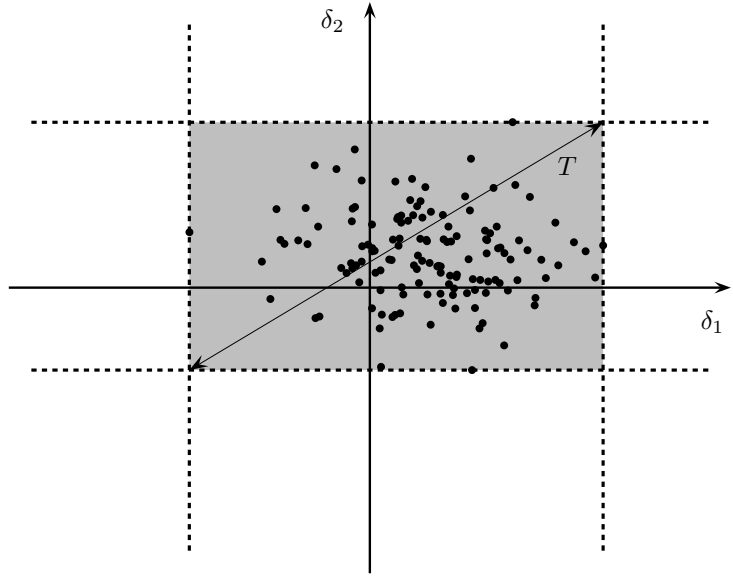

\subsection{Solution via Scenario Approach}

By inspection, each of the chance constraints $i=1,...,n$ has support rank $\sri=2$, because it only
involves the two variables $z_{i}$ and $t_{i}$.
For a fixed confidence level, \eg $\theta=10^{-6}$, the implicit sample sizes $K_{1},...,K_{n}$ in
\eqref{Equ:Sampling} can be computed for given values of $n$ and $\ep_{1},...,\ep_{n}\in (0,1)$ by
a bisection-based algorithm (see Section \ref{Sec:Chernoff}).
For simplicity, all $\ep_{1}=...=\ep_{n}$ are selected as equal, and since $\sr_{1}=...=\sr_{N}
=2$, the implicit sample sizes $K_{1}=...=K_{n}$ are also identical.

Given the outcomes of all multi-samples, the $\DSP$ is easily solved by the smallest $n$-cuboid that
contains all sampled points; see also Figure \ref{Fig:ExampleRP}. In other words, here the $\DSP$
has an analytic solution.

Table \ref{Tab:SampleSize}(a) summarizes the implicit sample sizes required for guaranteeing various
chance constraint levels $\ep_{i}$ in various dimensions $n$ (all with $\theta=10^{-6}$).
These sample sizes are also compared to those from the classic scenario approach, based on a
reformulation of \eqref{Equ:ExaConvProg} as an $\SCP$ according to the procedure outlined in Section
\ref{Sec:MCP}.

Observe from Table \ref{Tab:SampleSize} that the $\SCP$-based sample sizes are always larger than
those using the extensions of the $\MCP$ theory.
This effect increases, in particular, as the dimension $n$ of the optimization space grows larger.
The reason is that the support dimension of each chance constraint remains constant for all $n$,
whereas Helly's dimension grows as it equals to $n$.
The marginal growth of the sample size of the $\MCP$, despite the support rank $\sr_{i}=2$ being
constant, is the result of adjusting the confidence level $\theta$ to be (evenly) distributed among
the chance constraints, \ie $\theta_{i}=\theta/n$ for all $i=1,...,n$.

\begin{table}[H]
    \renewcommand\arraystretch{1.1}
    \centering
    \subfloat[b][$\MCP$-based Scenario Approach.]{
    \begin{tabular}{cr|rrrrrrr}
        \multicolumn{2}{c|}{\esp\esp\multirow{2}{1.4cm}{sample size $K_{i}$}}
        &\multicolumn{7}{c}{cuboid dimension $n=$}\\
        & & \parbox[l][4mm]{11mm}{\hfill 2\:\:\:}  & \parbox[l][4mm]{11mm}{\hfill 3\:\:\:} &
            \parbox[l][4mm]{11mm}{\hfill 5\:\:\:}  & \parbox[l][4mm]{11mm}{\hfill 10\:\:} &
            \parbox[l][4mm]{11mm}{\hfill 50\:\:} & \parbox[l][4mm]{11mm}{\hfill 100\:} &
            \parbox[l][4mm]{11mm}{\hfill 500\:}\\
        \hline
        \multirow{4}{*}{$\ep_{i}=$}
        &  1\%  & 1,734  & 1,777  & 1,831  & 1,903  & 2,072  & 2,144  & 2,311\\
        &  5\%  &   341  &   349  &   360  &   374  &   407  &   421  &   454\\
        & 10\%  &   166  &   170  &   176  &   182  &   199  &   205  &   221\\
        & 25\%  &    62  &    63  &    65  &    67  &    73  &    76  &    82
    \end{tabular}}\\
    \subfloat[b][$\SCP$-based Scenario Approach.]{
    \begin{tabular}{cr|rrrrrrr}
        \multicolumn{2}{c|}{\esp\esp\multirow{2}{1.4cm}{sample size $K_{i}$}}
        &\multicolumn{7}{c}{cuboid dimension $n=$}\\
        & & \parbox[l][4mm]{11mm}{\hfill 2\:\:\:}  & \parbox[l][4mm]{11mm}{\hfill 3\:\:\:} &
            \parbox[l][4mm]{11mm}{\hfill 5\:\:\:}  & \parbox[l][4mm]{11mm}{\hfill 10\:\:} &
            \parbox[l][4mm]{11mm}{\hfill 50\:\:} & \parbox[l][4mm]{11mm}{\hfill 100\:} &
            \parbox[l][4mm]{11mm}{\hfill 500\:}\\
        \hline
        \multirow{4}{*}{$\ep_{i}=$}
        &  1\%  & 2,334  & 2,722  & 3,431  & 5,020  & 15,588  & 27,535  & 115,786\\
        &  5\%  &   459  &   536  &   677  &   992  &  3,095  &  5,477  &  23,093\\
        & 10\%  &   225  &   263  &   332  &   488  &  1,533  &  2,719  &  11,506\\
        & 25\%  &    84  &    99  &   125  &   186  &    595  &  1,063  &   4,550
    \end{tabular}}
    \caption{Implicit sample sizes $K_{1}=...=K_{n}$ for the $\MCP$-based and the $\SCP$-based
    scenario approach, assuming a confidence level of $\theta=10^{-6}$, for varying problem
    dimension $n$ and chance constraint levels $\ep_{1}=...=\ep_{n}$.\label{Tab:SampleSize}}
    \renewcommand\arraystretch{1.0}
\end{table}

The larger sample size of the $\SCP$-based approach, as compared to the $\MCP$-based approach,
implies higher data requirements and higher computational efforts, but it also increases the
conservatism of the scenario solution.
The latter effect is quantified in Table \ref{Tab:RelObjValue}, showing the relative excess of the
(average) objective function values of the $\SCP$-based solutions over those of the $\MCP$-based
solutions.
Note that the objective values achieved by the $\SCP$-based approach are always higher than those
achieved by the $\MCP$-based approach, with the effect becoming increasingly significant as the
dimension $n$ of the decision space grows larger.

\begin{table}[H]
    \renewcommand\arraystretch{1.1}
    \centering
    \begin{tabular}{cr|rrrrrrr}
        \multicolumn{2}{c|}{\multirow{2}{1.6cm}{\,\, relative obj. value}}
        &\multicolumn{7}{c}{cuboid dimension $n=$}\\
        & & \parbox[l][4mm]{11mm}{\hfill 2\:\:\:}  & \parbox[l][4mm]{11mm}{\hfill 3\:\:\:} &
            \parbox[l][4mm]{11mm}{\hfill 5\:\:\:}  & \parbox[l][4mm]{11mm}{\hfill 10\:\:} &
            \parbox[l][4mm]{11mm}{\hfill 50\:\:} & \parbox[l][4mm]{11mm}{\hfill 100\:} &
            \parbox[l][4mm]{11mm}{\hfill 500\:}\\
        \hline
        \multirow{4}{*}{$\ep_{i}=$}
        &  1\%  &  2.4\%  & 3.4\%  &  5.0\% &  7.5\% & 14.8\% & 18.4\% & 26.9\%\\
        &  5\%  &  3.3\%  & 4.6\%  &  6.6\% &  9.8\% & 19.3\% & 23.8\% & 34.4\%\\
        & 10\%  &  3.9\%  & 5.4\%  &  7.6\% & 11.5\% & 22.2\% & 27.4\% & 39.3\%\\
        & 25\%  &  5.0\%  & 7.2\%  & 10.1\% & 15.1\% & 28.5\% & 34.7\% & 49.1\%
    \end{tabular}
    \caption{Objective function value of $\SCP$-based scenario solution as a percentage increase
    over the $\MCP$-based scenario solution, based on the sample sizes in Table \ref{Tab:SampleSize}
    and a multi-variate standard normal distribution for $\den$. Each of the indicated values
    represents an average over one million simulation runs.\label{Tab:RelObjValue}}
    \renewcommand\arraystretch{1.0}
\end{table}

\section*{Acknowledgement}

The research of L.~Fagiano has received funding from the European Union Seventh Framework
Programme FP7/2007-2013 under grant agreement number PIOF-GA-2009-252284, Marie Curie project
`Innovative Control, Identification and Estimation Methodologies for Sustainable Energy
Technologies'.

\begin{appendix}

\section{Probability Distributions}\label{Sec:ProbDistr} 

Several basic probability-related functions are used throughout this paper. 
The \emph{Binomial Distribution Function} \cite[p.\,26.1.20]{Abramowitz:1970}
\begin{equation}\label{Equ:BinDistr}
    \Phi(x;K,\ep):=\sum_{j=0}^{x}{K\choose j}\ep^{j}(1-\ep)^{K-j}
\end{equation}
expresses the probability of seeing at most $x\in\BN_{0}^{K}$ successes in $K\in\BN$ independent
Bernoulli trails, where the probability of success is $\ep\in(0,1)$ per trial.
The (real) \emph{Beta Function} \cite[p.\,6.2.1]{Abramowitz:1970}
\begin{equation}\label{Equ:BetaFct1}
	\B(a,b):=\int_{0}^{1}\xi^{a-1}(1-\xi)^{b-1}\di\ep
\end{equation}
is defined for any parameters $a,b\in\BR_{+}$, and $\xi\in(0,1)$; it also satisfies the identity 
\cite[p.\,6.2.2]{Abramowitz:1970}
\begin{equation}\label{Equ:BetaFct2}
	\B(a,b)=\B(b,a)=\frac{\Gamma(a)\Gamma(b)}{\Gamma(a+b)}\ec
\end{equation}
where $\Gamma:\BR_{+}\to\BR_{+}$ denotes the (real) \emph{Gamma Function} with $\Gamma(n+1)=n!$ for
any $n\in\BN_{0}^{\infty}$ \cite[p.\,6.1.5]{Abramowitz:1970}.
The corresponding \emph{Incomplete Beta Function} \cite[p.\,6.6.1]{Abramowitz:1970} is then given by
\begin{equation}\label{Equ:IncBetaFct1}
	\B(\ep;a,b):=\int_{0}^{\ep}\xi^{a-1}(1-\xi)^{b-1}\di\xi
	            =\int_{1-\ep}^{1}\xi^{b-1}(1-\xi)^{a-1}\di\xi\ec
\end{equation}
where the last equality follows by a simple substitution. An important identity is obtained from
\cite[pp.\,3.1.1,\,6.6.2,\,26.5.7]{Abramowitz:1970},
\begin{equation}\label{Equ:IncBetaFct2}
	\B(\ep;a,b)=\B(a,b)\sum_{j=a}^{a+b-1}{a+b-1 \choose j}\ep^{j}(1-\ep)^{a+b-1-j}\ec
\end{equation}
which can written more compactly by use of the binomial distribution \eqref{Equ:BinDistr}, see for
instance \cite[p.\,3437]{Cala:2010}:
\begin{equation}\label{Equ:IncBetaFct3}
	\B(\ep;a,b)=\frac{1}{b}{a+b-1\choose b}^{-1}\Phi(b-1;a+b-1,1-\ep)\ef
\end{equation}

\end{appendix}

\bibliographystyle{siam}
\bibliography{paperbib}

\end{document}